\documentclass[12pt,a4paper,hidelinks]{amsart}  
\usepackage{graphicx,amssymb,stmaryrd,xspace}
\RequirePackage[raiselinks=false,colorlinks=false,citecolor=blue,urlcolor=black,linkcolor=blue,bookmarksopen=true]{hyperref}
\newcommand{\arxiv}[1]{\href{http://arxiv.org/pdf/#1}{arXiv:#1}}
\newcommand{\doi}[1]{\href{http://doi.org/#1}{DOI:#1}}
\usepackage[margin=2.8cm]{geometry}			     

\usepackage{eucal,marginnote}
\def\ourmargin#1{\marginnote{\textcolor{red}{\footnotesize\sc #1}}}

\DeclareMathAlphabet{\mathbbe}{U}{bbold}{m}{n}
\newcommand{\simplexcategory}{\mathbbe{\Delta}}

\usepackage{tikz}
\usetikzlibrary{matrix,arrows}
\usepackage{tikz-cd}

\makeatletter

\renewenvironment{abstract}{%
  \ifx\maketitle\relax
    \ClassWarning{\@classname}{Abstract should precede
      \protect\maketitle\space in AMS document classes; reported}%
  \fi
  \global\setbox\abstractbox=\vtop \bgroup
    \normalfont\Small
    \list{}{\labelwidth\z@
      \leftmargin2pc \rightmargin\leftmargin
      \listparindent\normalparindent \itemindent\z@
      \parsep\z@ \@plus\p@
      
    }%
    \item[\hskip\labelsep\scshape\abstractname.]%
}{%
  \endlist\egroup
  \ifx\@setabstract\relax \@setabstracta \fi
}
\def\@setabstract{\@setabstracta \global\let\@setabstract\relax}
\def\@setabstracta{%
  \ifvoid\abstractbox
  \else
    \skip@20\p@ \advance\skip@-\lastskip
    \advance\skip@-\baselineskip \vskip\skip@
    \box\abstractbox
    \prevdepth\z@ 
  \fi
}

\makeatother

\usepackage[v2,all]{xy}
\newcommand{\ulpullback}[1][ul]{\save*!/#1-4ex/#1:(-1,1)@^{|-}\restore}
\newcommand{\dlpullback}[1][dl]{\save*!/#1-4ex/#1:(-1,1)@^{|-}\restore}

\newcommand{\drpullback}[1][dr]{\save*!/#1-4ex/#1:(-1,1)@^{|-}\restore}

\usepackage{ stmaryrd }

\def\Map{{\mathrm{Map}}}

\newcommand{\relfin}{\mathrm{\,rel.fin.}}
\newcommand{\finsup}{\mathrm{\,fin.sup.}}

\def\M{M\"obius\xspace}
\newcommand{\FILT}{tight\xspace}
\newdir{ >}{{}*!/-7pt/@{>}}

\def\overarrow#1{{\vec{#1}}}
\def\nondeg{\overarrow}

\newcommand{\Phieven}{\Phi_{\text{\rm even}}}
\newcommand{\Phiodd}{\Phi_{\text{\rm odd}}}

\newcommand{\un}{\underline}
\newcommand{\Deltagen}{\simplexcategory_{\text{\rm active}}}
\newcommand{\Deltainj}{\simplexcategory_{\text{\rm inj}}}

\newcommand{\culf}{CULF\xspace}

\renewcommand{\Im}{\operatorname{Im}}
\newcommand{\Nat}{\operatorname{Nat}}
\newcommand{\Cons}{\operatorname{Cons}}

\def\GG{\mathbf{G}}

\providecommand{\norm}[1]{\left| {#1}\right|}

\def\into{\hookrightarrow}
\newcommand{\shortsetminus}{\,\raisebox{1pt}{\ensuremath{\mathbb r}\,}}

\newcommand{\pro}[1]{\underleftarrow{#1}}
\newcommand{\ind}[1]{\underrightarrow{#1}}
\providecommand{\kat}[1]{\text{\textbf{\textsl{#1}}}}

\newcommand{\LIN}{\kat{LIN}}
\newcommand{\lin}{\kat{lin}}

\newcommand{\upperstar}{^{\raisebox{-0.25ex}[0ex][0ex]{\(\ast\)}}}

\newcommand{\lowershriek}{_!}

\newcommand{\isopil}{\stackrel{\raisebox{0.1ex}[0ex][0ex]{\(\sim\)}}%
			{\raisebox{-0.15ex}[0.28ex]{\(\rightarrow\)}}}

\newcommand{\tensor}{\otimes}
\newcommand{\op}{^{\text{{\rm{op}}}}}

\newcommand{\Set}{\kat{Set}}
\newcommand{\Grpd}{\mathcal{S}}

\newcommand{\grpd}{\mathcal{F}}

\newcommand{\N}{\mathbb{N}}
\newcommand{\Q}{\mathbb{Q}}

\newcommand{\CC}{\mathcal{C}}

\def\rTo{\longrightarrow}

\newcommand{\name}[1]{\ulcorner #1\urcorner}

\newcommand{\Fun}{\operatorname{Fun}}
\newcommand{\Aut}{\operatorname{Aut}}
\newcommand{\id}{\operatorname{id}}

\makeatletter
\def\subsection{\@startsection{subsection}{2}%
  \z@{.5\linespacing\@plus.7\linespacing}{.5\linespacing}%
  {\normalfont\bfseries}}
\makeatother

\newtheorem{lemma}{Lemma}[section]
\newtheorem{prop}[lemma]{Proposition}
\newtheorem{thm}[lemma]{Theorem}
\newtheorem{theorem}[lemma]{Theorem}
\newtheorem{cor}[lemma]{Corollary}
\theoremstyle{definition}
\newtheorem{eks}[lemma]{Example}
\newtheorem{BM}[lemma]{Remark}

\newtheorem{taller}[lemma]{$\!\!$}

\newenvironment{blanko}[1]%
{\begin{taller}{\normalfont\bfseries  #1}\normalfont}%
{\end{taller}}

\newenvironment{blanko*}[1]{\begin{list}{\bf {#1} }%
{\setlength{\labelsep}{0mm}\setlength{\leftmargin}{0mm}%
\setlength{\labelwidth}{0mm}\setlength{\listparindent}{\parindent}%
\setlength{\parsep}{\parskip}\setlength{\partopsep}{0mm}}%
\item%
}{\end{list}}

\newenvironment{deff}%
{\begin{list}{\em Definition. }%
{\setlength{\labelsep}{0mm}\setlength{\leftmargin}{0mm}%
\setlength{\labelwidth}{0mm}\setlength{\listparindent}{\parindent}%
\setlength{\parsep}{\parskip}\setlength{\partopsep}{0mm}}%
\item}{\end{list}}

\newenvironment{proof*}[1]{\begin{list}{\em #1 }%
{\setlength{\labelsep}{0mm}\setlength{\leftmargin}{0mm}%
\setlength{\labelwidth}{0mm}\setlength{\listparindent}{\parindent}%
\setlength{\parsep}{\parskip}\setlength{\partopsep}{0mm}}%
\item}{\qed\end{list}}

\thanks{%
  The first author 
  was partially supported by grants 
  MTM2012-38122-C03-01,  
  MTM2013-42178-P,       
  2014-SGR-634,          
  MTM2015-69135-P,   
  MTM2016-76453-C2-2-P  (AEI/FEDER, UE), and   
  2017-SGR-932,          
  the second author 
  by 
  MTM2013-42293-P, 
  MTM2016-80439-P  (AEI/FEDER, UE),
  and
  2017-SGR-1725,
  and the third author   
  by
  MTM2013-42178-P and
  MTM2016-76453-C2-2-P (AEI/FEDER, UE)}

\author{Imma G\'alvez-Carrillo}
\address{Departament de Matem\`atiques
      \\Universitat Polit\`ecnica de Catalunya
	  }
\email{m.immaculada.galvez@upc.edu}

\author{Joachim Kock}
\address{Departament de Matem\`atiques
       \\Universitat Aut\`onoma de Barcelona
	   }
\email{kock@mat.uab.cat}

\author{Andrew Tonks}
\address{Department of Mathematics\\ 
University of Leicester
}
\email{apt12@le.ac.uk}

\title[Decomposition spaces and M\"obius inversion]{Decomposition spaces, incidence algebras and 
M\"obius inversion II: completeness, length filtration, and finiteness}
\date{}        

\begin{document}

\vspace*{-8mm}

\textcolor{red}{\hspace*{-1.2em}\fbox{ 
\begin{minipage}{150mm}
NOTE: This version corrects a mistake in the published version 
(Adv.~Math.~2018).\linebreak
The changes are in 7.12--7.17.
\end{minipage}
}
}

\vspace*{6mm}

\begin{abstract}
  This is the second in
  a trilogy of papers
  introducing and studying the notion of
  decomposition space as a general framework for incidence algebras and \M
  inversion, with coefficients in $\infty$-groupoids.  A decomposition space is
  a simplicial $\infty$-groupoid satisfying an exactness condition weaker than
  the Segal condition.  Just as the Segal condition expresses
  composition, the new condition expresses decomposition.

  In this paper, we introduce various technical conditions on decomposition 
  spaces.  
  The first is a completeness condition (weaker than Rezk completeness),
  needed to control
  simplicial
  nondegeneracy.  For complete decomposition spaces
  we establish a general \M inversion principle, expressed as an explicit 
  equivalence of $\infty$-groupoids.
  Next we analyse two finiteness conditions on decomposition spaces.  The first,
  that of locally finite length, guarantees the existence of the important length
  filtration for the associated incidence coalgebra.  We show that a
  decomposition space of locally finite length is actually the left Kan extension
  of a semi-simplicial space.
  The second finiteness condition, local finiteness, ensures we can take
  homotopy cardinality to pass from the level of $\infty$-groupoids to the level
  of $\Q$-vector spaces.

  These three conditions --- completeness, locally finite length, and local finiteness 
  ---
  together define our notion of
  \emph{\M decomposition space},
  which extends Leroux's notion of \M category (in turn a common 
  generalisation of
  the locally finite posets of Rota et al.\ and of the finite
  decomposition monoids of Cartier--Foata), but which also covers many
  coalgebra constructions which do not arise from \M categories, such
  as the Fa\`a di Bruno and Connes--Kreimer bialgebras.

  Note: The notion of decomposition space was arrived at independently by
  Dyckerhoff and Kapranov (arXiv:1212.3563) who call them unital $2$-Segal
  spaces.
\end{abstract}



\maketitle

\vspace*{-8mm}

\small

\tableofcontents

\vspace*{-16mm}

\normalsize

\setcounter{section}{-1}

\addtocontents{toc}{\protect\setcounter{tocdepth}{1}}

\section{Introduction}

In the first paper of this trilogy~\cite{GKT:DSIAMI-1},
we introduced the notion of
decomposition space as a general framework for incidence (co)algebras.
It is equivalent to the notion of unital $2$-Segal space of Dyckerhoff
and Kapranov~\cite{Dyckerhoff-Kapranov:1212.3563}.
The
relevant main results are recalled in Section \ref{sec:prelims} below.  A
decomposition space is a simplicial $\infty$-groupoid $X$ satisfying a certain
exactness condition, weaker than the Segal condition.  Just as the Segal
condition expresses
composition, the new condition expresses
decomposition, 
and implies the existence of an 
incidence (co)algebra.  There is a rich supply of examples in
combinatorics~\cite{GKT:ex}.  An easy example is the 
decomposition space of graphs (yielding the chromatic Hopf 
algebra~\cite{Schmitt:hacs}), which will serve as a running example.
In the present paper we proceed to 
establish a \M inversion principle for what we call {\em complete}
decomposition spaces, and analyse the associated finiteness issues.

Classically~\cite{Rota:Moebius}, the \M inversion principle states that the zeta
function
of any incidence algebra (of a locally finite poset, say, or more
generally a \M category in the sense of Leroux~\cite{Leroux:1975}) is invertible
for the convolution product; its inverse is by definition the \M function.  The
\M inversion formula is a powerful and versatile counting device.
Since it
is an equality stated at the vector-space level in the incidence algebra, it
belongs to algebraic combinatorics rather than bijective combinatorics.
It is possible to give \M inversion a bijective meaning, by following
the objective method, pioneered in this context by Lawvere and
Menni~\cite{LawvereMenniMR2720184}, which seeks to lift algebraic identities to
the `objective level' of (finite) sets and bijections, working with certain
categories spanned by the combinatorial objects instead of  with vector
spaces spanned by isoclasses of these objects.  The algebraic identity then
appears as the cardinality of the bijection established at the objective level.

To illustrate the objective viewpoint, observe that a vector in the free vector
space on a set $B$ is just a collection of scalars indexed by (a finite subset
of) $B$.  The objective counterpart is a family of sets indexed by $B$, i.e.~an
object in the slice category $\Set_{/B}$.  `Linear maps' at this level are given
by spans $A \leftarrow M \to B$.  The \M inversion principle states an equality
between certain linear maps (elements in the incidence algebra).  At the
objective level, such an equality can be expressed as a bijection between sets
in the spans representing those linear functors.  In this way, the algebraic
identity is revealed to be just the cardinality of a bijection of sets, which
carry much more structural information.
As an example, the objective
counterpart of the binomial algebra is the category of species with the Cauchy
tensor product~\cite{GKT:ex}, a much richer structure, and at the objective
level there are obstructions to cancellations in the M\"obius function that take
place at the numerical level only.  The significance of these phenomena is not
yet clear, and is under investigation~\cite{GKT:ex}.
Lawvere and Menni~\cite{LawvereMenniMR2720184} established an objective version of the \M inversion principle for \M categories in the sense of Leroux~\cite{Leroux:1975}.

Our discovery in \cite{GKT:DSIAMI-1} is that
something considerably weaker than a category suffices to construct an incidence
algebra, namely a decomposition space.  This discovery is interesting even at
the level of simplicial sets, but we work at the level of simplicial $\infty$-groupoids.  Thus, the
role of vector spaces is played by slices of the $\infty$-category of
$\infty$-groupoids.  In~\cite{GKT:HLA} we have developed the necessary `homotopy
linear algebra' and homotopy cardinality, extending and streamlining many results of
Baez--Hoffnung--Walker~\cite{Baez-Hoffnung-Walker:0908.4305} who worked with
$1$-groupoids.

The decomposition-space axiom on a simplicial
$\infty$-groupoid $X$ is expressly the condition needed for a canonical
coalgebra structure to be induced on the slice $\infty$-category $\Grpd_{/X_1}$, 
(where $\Grpd$ denotes the $\infty$-category of $\infty$-groupoids, also called spaces).
The comultiplication is the linear functor
$$
\Delta : \Grpd_{/X_1} \to \Grpd_{/X_1} \tensor  \Grpd_{/X_1}
$$
given by the span
$$
X_1 \stackrel{d_1}\longleftarrow X_2 \stackrel{(d_2,d_0)}\longrightarrow 
X_1\times X_1.
$$
This can be read as saying that comultiplying an edge $f\in X_1$ returns the sum
of all pairs of edges $(a,b)$ that are the short edges of a $2$-simplex with long
edge $f$.  In the case that $X$ is the nerve of a category, this is the sum of
all pairs $(a,b)$ of arrows with composite $b\circ a=f$.

The aims of this paper are to establish a \M inversion principle in the
framework of \emph{complete} decomposition spaces, and also to introduce the
necessary \emph{finiteness} conditions on a complete decomposition space to
ensure that incidence (co)algebras and \M inversion descend to classical
vector-space-level coalgebras on taking the homotopy cardinality of the objects
involved.  Along the way we also establish some auxiliary results of a
more technical nature which are needed in the applications in the sequel
papers~\cite{GKT:MI,GKT:restriction,GKT:ex}. 

\medskip

\medskip

We proceed to summarise the main results.

\medskip

After briefly reviewing in Section~\ref{sec:prelims} the notion of
decomposition space and the notion of {\culf} maps between them ---
simplicial maps that induce coalgebra homomorphisms
--- we come 
to the notion of completeness in Section~\ref{sec:complete}:

\begin{deff}
  We say that a decomposition space $X$ is {\em complete} 
  (\ref{complete})
  when $s_0: X_0 \to
  X_1$ is a monomorphism.  It then follows that all degeneracy
  maps are monomorphisms (Lemma \ref{all-s-mono}).
\end{deff}
The motivating feature of this notion is that all issues concerning degeneracy
can then be settled in terms of the canonical projection maps $X_r \to (X_1)^r$
sending a simplex to its principal edges: a simplex in a complete decomposition
space is nondegenerate precisely when all its principal edges are nondegenerate
(Corollary \ref{effective=nondegen}).  Let $\nondeg X_r \subset X_r$ denote the
subspace of these nondegenerate simplices.

\smallskip

For any decomposition space $X$, the comultiplication on $\Grpd_{/X_1}$ yields a
convolution product on the linear dual $\Grpd^{X_1}$
(that is, the category of linear functors from $\Grpd_{/X_1}$ to $\Grpd$) 
called the {\em incidence
algebra} of $X$.  This contains, in particular, the {\em zeta functor} $\zeta$,
given by the span $X_1 \stackrel=\leftarrow X_1 \to 1$, and the counit
$\varepsilon$ (neutral for convolution) given by $X_1 \leftarrow X_0 \to
1$. 
In a complete decomposition space $X$ we can consider the spans $X_1 \leftarrow
\nondeg X_r \to 1$ and the linear functors $\Phi_r$ they define in the incidence
algebra of $X$.  We can now establish the decomposition-space version of the
\M inversion principle, in the spirit of \cite{LawvereMenniMR2720184}:

\smallskip

\noindent
{\bf Theorem~\ref{thm:zetaPhi}.} {\em
For a complete decomposition space, there are explicit equivalences
$$
\zeta * \Phieven
\;\;\simeq\;\; \varepsilon\;\; +\;\; \zeta * \Phiodd
,\qquad\qquad
\Phieven *\zeta \;\;\simeq\;\; \varepsilon \;\;+ \;\; \Phiodd*\zeta.
$$
}

It is tempting to read this as saying that ``$\Phieven -\Phiodd$'' is the
convolution inverse of $\zeta$, but the lack of additive inverses in $\Grpd$
necessitates our sign-free formulation.  Upon taking homotopy cardinality, as we
will later, this yields the usual \M inversion formula $\mu= \Phieven-\Phiodd$,
valid in the incidence algebra with $\Q$-coefficients.  

\smallskip

Having established the general \M inversion principle on the objective
level, we proceed to analyse the finiteness conditions
on complete decomposition spaces
needed for this
principle to descend to the vector-space level of $\Q$-algebras.
There are two conditions: $X$ should be of locally finite length 
(Section~\ref{sec:length}), and $X$ should be locally finite
(Section~\ref{sec:findec}).  The first is a numerical condition, like a chain 
condition; the second is a homotopy finiteness condition.
Complete 
decomposition spaces satisfying both conditions are called {\em \M decomposition 
spaces} (Section~\ref{sec:M}).  We analyse the two conditions separately.

\begin{deff} (Cf.~\ref{length}.)
  The {\em length} of an arrow $f$ is the greatest dimension of a nondegenerate
  simplex with long edge $f$.  We say that a complete decomposition space is
  {\em of locally finite length} --- we also say {\em tight} --- when every
  arrow has finite length.
\end{deff}

Although many
examples coming from combinatorics do satisfy this condition, it is actually a 
rather strong condition, as witnessed by the following result,
which is a consequence of Propositions~\ref{prop:splitKan} and 
\ref{prop:tight=>split}:

\medskip
\noindent {\em Every tight decomposition space 
is the left Kan extension of a semi-simplicial space.}

\medskip

We can prove this result for more general simplicial spaces, and
digress to establish this in Section~\ref{sec:split}: 
we say a complete simplicial space is {\em split} if all face maps preserve
nondegenerate simplices. 
In Corollary~\ref{cor:split=IU} we show this is the
analogue of the condition for categories that identities are indecomposable,
enjoyed in particular by \M categories in the sense of Leroux~\cite{Leroux:1975}.
We prove that a simplicial space is split if and only if it is the left Kan
extension along $\Deltainj \subset \simplexcategory$ of a semi-simplicial space
$\Deltainj\op\to\Grpd$, and in fact we establish more precisely:

\medskip

\noindent
{\bf Theorem~\ref{thm:semisimpl=splitcons}.} {\em
Left Kan extension along $\Deltainj\subset \simplexcategory$
induces an equivalence of $\infty$-categories
$$
\Fun(\Deltainj\op,\Grpd) \simeq \kat{Split}^{\mathrm cons},
$$
where the right-hand side is
the $\infty$-category of split simplicial spaces and conservative maps.}

\medskip

This has the following interesting corollary.

\medskip

\noindent
{\bf Proposition~\ref{prop:split=semi}.}
{\em 
  Left Kan extension along $\Deltainj\subset \simplexcategory$ induces
  an equivalence between the $\infty$-category of 
  $2$-Segal semi-simplicial spaces and ULF maps, and the $\infty$-category of
  split decomposition spaces and \culf maps.
}
\medskip

We show that a complete decomposition space $X$ is tight if and only if it has a filtration
$$
X_\bullet^{(0)} \into X_\bullet^{(1)} \into \cdots \into X
$$
of \culf monomorphisms, the so-called {\em length filtration}.
This is precisely the structure needed to get a filtration of the incidence
coalgebra (\ref{coalgebrafiltGrpd}).

\bigskip

In Section~\ref{sec:findec} we impose the finiteness condition needed to be able
to take homotopy cardinality and obtain coalgebras and algebras at the numerical level 
of $\Q$-vector spaces (and profinite-dimensional $\Q$-vector spaces).

\begin{deff}
An $\infty$-groupoid $S$ is {\em locally finite} if at each base point $x$ the
  homotopy groups $\pi_i (S,x)$ are finite for $i\geq1$ and are trivial for $i$
  sufficiently large.  It is called {\em finite} if furthermore it has only
  finitely many components. A map of $\infty$-groupoids is called a {\em finite map} if its fibres are finite $\infty$-groupoids.

  A decomposition space $X$ is called {\em locally finite} (\ref{finitary}) when
  $X_1$ is a locally finite $\infty$-groupoid
  and $s_0:X_0\to X_1$ and $d_1: X_2 \to X_1$ are finite maps.
\end{deff}

The condition `locally finite'
extends the notion of locally finite for posets.
The condition
ensures that the coalgebra structure descends to 
coefficients in finite $\infty$-groupoids, and hence, 
via homotopy cardinality, to
$\Q$-algebras.  In Section~\ref{sec:tioncoeff} we calculate the section 
coefficients (structure constants for the (co)multiplication) in some
easy cases. 

Finally we introduce the \M condition: 
\begin{deff}
  A complete decomposition space is called {\em \M} (\ref{M})
  when it is locally finite and of locally finite length (i.e.~is tight).
\end{deff}
These are the conditions needed for the general \M inversion formula
to descend to coefficients in finite $\infty$-groupoids
and hence $\Q$-coefficients, giving the following formula
for the \M function (convolution inverse to the zeta function):
$$
\norm{\mu} = \norm{\Phieven}-\norm{\Phiodd} .
$$

\medskip

We have strived throughout to distill the most natural
conditions from the requirements imposed by applications of the theory,
and we find it an attractive feature that all
the conditions can be formulated categorically.
Just as the decomposition-space axiom is an exactness
condition (that certain `active-inert'
pushouts in $\simplexcategory$
are taken to pullbacks), it is noteworthy
that further conditions we require
--- completeness, stiffness, indecomposable units, and
splitness --- are also exactness conditions (stipulating that
certain other classes of pushouts are taken to pullbacks,
cf.~\ref{completeexact},
\ref{stiff}, 
\ref{IU}, 
and Corollary~\ref{cor:splitexact}).
This fact is both conceptually pleasing and
facilitates efficient arguments. 

\medskip

\noindent {\bf Related work.}
The notion of decomposition space was discovered independently by Dyckerhoff 
and Kapranov~\cite{Dyckerhoff-Kapranov:1212.3563}, who call them unital 
$2$-Segal spaces.  While some of the basic results in \cite{GKT:DSIAMI-1}
were also proved in \cite{Dyckerhoff-Kapranov:1212.3563}, the present paper
has no overlap with \cite{Dyckerhoff-Kapranov:1212.3563}.

The results in this paper on \M inversion are in the tradition of Leroux et
al.~\cite{Leroux:1975}, \cite{Content-Lemay-Leroux}, \cite{Leroux:IJM},
D\"ur~\cite{Dur:1986}, and Lawvere--Menni~\cite{LawvereMenniMR2720184}.  There is
a different notion of \M category, due to Haigh~\cite{Haigh}.  The two notions
have been compared, and to some extent unified, by
Leinster~\cite{Leinster:1201.0413}, who calls Leroux's \M inversion {\em fine}
and Haigh's {\em coarse} (as it only depends on the underlying graph of the
category).  We should mention also the $K$-theoretic \M inversion for
quasi-finite EI categories of L\"uck and collaborators \cite{Luck:1989},
\cite{Fiore-Luck-Sauer:0908.3417}.

\medskip

\noindent
{\bf Note.}
  This paper is the second in a series, originally posted on the arXiv 
  as a single manuscript {\em Decomposition spaces, incidence algebras 
  and \M inversion}~\cite{GKT:1404.3202} but split for publication into:
  \begin{enumerate}
  \item[(0)] 
    Homotopy linear algebra
    \cite{GKT:HLA}
    
  \item 
    Decomposition spaces, incidence algebras and \M inversion I:
    basic theory
  \cite{GKT:DSIAMI-1}
    \item Decomposition spaces, incidence algebras and \M inversion II:
    completeness, length filtration, and finiteness
    [this paper]

  \item
  Decomposition spaces, incidence algebras and M\"obius inversion III:
  the decomposition space of \M intervals
    \cite{GKT:MI}

  \item 
    Decomposition spaces and restriction species
    \cite{GKT:restriction}

    \item 
    Decomposition spaces in combinatorics
    \cite{GKT:ex}.

  \end{enumerate}

\medskip

\noindent
{\bf Acknowledgments.}
  This work has been influenced very much by Andr\'e Joyal,
  whom we thank for enlightening discussions and advice, and specifically
  for suggesting to us to investigate the notion of split decomposition
  spaces.  We also thank Louis Carlier and Alex Cebri\'an for useful feedback,
  and the referee for pertinent suggestions, which led to improved exposition.

\section{Preliminaries on decomposition spaces}
\label{sec:prelims}

\begin{blanko}{$\infty$-groupoids.}
  We work in the $\infty$-category $\Grpd$ of $\infty$-groupoids, also called \emph{spaces}, and in closely
  related $\infty$-categories such as its slices.  By $\infty$-category we mean
  quasi-category in the sense of Joyal~\cite{Joyal:qCat+Kan}, \cite{Joyal:CRM},
  but follow rather the terminology of Lurie~\cite{Lurie:HTT}.
  Most of our arguments are elementary, though, and for this
  reason we can get away with model-independent reasoning
  rather than working with the Joyal model structure on simplicial 
  sets.  In particular, when we refer to the $\infty$-category
  $\Grpd_{/B}$ (whose objects are maps of $\infty$-groupoids $X\to B$), we only refer to an
  $\infty$-category determined up to equivalence by a certain universal 
  property, and do not make any distinction between the specific models for
  this object exploited by Joyal and Lurie (normal slice and fat slice).
\end{blanko}

\begin{blanko}{Pullbacks and Fibres.}
  Pullbacks play an essential role in many of our arguments.
  By \emph{pullback} we always mean pullback in the $\infty$-category $\Grpd$.
  This notion enjoys a universal property which in the model-independent
  formulation is similar to the universal property of
  the pullback in ordinary categories (such as $\Set$).
  Again, we shall only ever
  need homotopy invariant properties, 
  making it irrelevant which particular model is chosen for the notion of
  pullback in the Joyal model structure for quasi-categories.
  In particular, the \emph{fibre} $X_b$ of a map $f:X\to B$ over a base point $b$ in $B$ is also a homotopy invariant notion: it is the pullback of $f$ along the map $\name b:1\to B$ that picks out the base point.
  
  One property which we shall use repeatedly is
  the following elementary lemma (a proof can be found in
  \cite[4.4.2.1]{Lurie:HTT}).
\end{blanko}

\begin{lemma}\label{pbk}
  In any diagram of $\infty$-groupoids
  $$
  \vcenter{\xymatrix{
   \cdot\dto \rto &  \cdot\dto \rto &  \cdot\dto \\
  \cdot\rto & \cdot \rto & \cdot 
  }}
  $$
  if the outer rectangle and the right-hand square are pullbacks,
  then the left-hand square is a pullback.
\end{lemma}

\begin{blanko}{Monomorphisms.}\label{def:mono}
  The homotopy invariant notion of monomorphism of $\infty$-groupoids plays an important
  role throughout this paper, notably through
  the definition of complete decomposition space~(\ref{complete}).
  A map of $\infty$-groupoids is a {\em monomorphism} when its
  (homotopy) fibres are $(-1)$-groupoids
(i.e.~are either empty or contractible).

  (We warn against a potential point of confusion:
  in the Joyal model, $\infty$-groupoid means Kan
  complex, but the homotopy-invariant notion of monomorphism
  between $\infty$-groupoids is {\em not} the same as
  levelwise injective simplicial map between Kan complexes.
For example, any equivalence of $\infty$-groupoids is a monomorphism,
but not every equivalence of Kan complexes is levelwise
injective.  Conversely the inclusion $1\to BG$ of a point into the
classifying space of a group is not a monomorphism of 
$\infty$-groupoids, but it is injective levelwise in the sense of Kan 
complexes.)

  In some respects, this notion of monomorphism does behave as for sets: for example, if $f:X\to Y$
  is a monomorphism, then there is a complement $Z:=Y\shortsetminus X$ such that
  $X + Z\simeq Y$.  Hence a monomorphism is essentially an equivalence
  from $X$ onto some connected components of $Y$.  On the other hand, 
  a crucial difference between sets and $\infty$-groupoids is that diagonal
  maps of $\infty$-groupoids are not in general monomorphisms.  In fact $X \to X \times 
  X$ is a monomorphism if and only if $X$ is discrete (i.e.~equivalent to a set).
\end{blanko}

\begin{blanko}{Linear algebra with coefficients in $\infty$-groupoids \cite{GKT:HLA}.}
  The slice $\infty$-categories of the form $\Grpd_{/I}$ form the objects
  of a symmetric monoidal $\infty$-category $\LIN$, described in
  detail in \cite{GKT:HLA}: the morphisms are the linear functors,
  meaning that they preserve homotopy sums, or equivalently indeed
  all colimits.  Such functors are given by spans: the span
  $$
  I \stackrel p \leftarrow M \stackrel q \to J$$
  defines the linear functor
  $$
  q\lowershriek \circ p\upperstar : \Grpd_{/S}  \longrightarrow  \Grpd_{/T} 
  $$
  given by pullback along $p$ followed by composition with $q$.
  The $\infty$-category $\LIN$ can play the role of the category of
  vector spaces, although to be strict about that interpretation,
  finiteness conditions should be imposed, as we do later in this paper 
  (Section~\ref{sec:findec}).
  
  The symmetric monoidal structure on $\LIN$ is easy to describe on objects:
  we have
  $$
  \Grpd_{/I} \tensor \Grpd_{/J} {}:={} \Grpd_{/I\times J},
  $$
  just as the tensor product of vector spaces with bases indexed by sets $I$ 
  and $J$ is the vector space with basis indexed by $I\times J$.
  The neutral object is $\Grpd_{/1}\simeq\Grpd$.
\end{blanko}

\begin{blanko}{Simplicial spaces.}
    Throughout, our main objects of study will be simplicial spaces
    $X:\simplexcategory\op\to\Grpd$,
    by which we mean objects in the functor $\infty$-category
    $\Fun(\simplexcategory\op,\Grpd)$.  A simplicial space (synonym for
    simplicial $\infty$-groupoid) is thus
    a homotopy-coherent simplicial diagram of $\infty$-groupoids.
    Note that this means that the simplicial identities are 
    squares that commute up to a homotopy, such as for example
    \begin{equation}\label{eq:decompex}\vcenter{
	\xymatrix{
	X_3 \ar[r]^{d_3} \ar[d]_{d_1} & X_2 \ar[d]^{d_1} \\
	X_2 \ar[r]_{d_2} & X_1,}
	}
    \end{equation}
    and it makes sense to ask whether such a square is a pullback.
    We shall never need to spell out the homotopies, as 
    only their structural properties are needed.  
    
    By an {\em $n$-simplex} of
    $X$ we mean an object in the $\infty$-groupoid $X_n$, which in
    turn can be described (via the Yoneda lemma for 
    $\infty$-groupoid-valued presheaves) as the mapping space 
    $\Map(\Delta[n],X)$. If $\sigma$ is an object of $X_n$ then we write $\name{\sigma}:1\to X_n$ for the corresponding map.

    A {\em simplicial map} $f: X \to Y$ between
  simplicial spaces $X$ and $Y$ is by definition
  an object in the mapping space $\Map_{\Grpd}(X,Y)$.
  It amounts to a sequence of maps $f_i : X_i \to Y_i$ commuting
  with the face and degeneracy maps up to specified
  coherent homotopies. 
\end{blanko}

We briefly review the main notions and results from 
the first paper in the trilogy~\cite{GKT:DSIAMI-1},
and in particular the notion of decomposition space.
This notion is equivalent to that of unital $2$-Segal space,
introduced by Dyckerhoff and Kapranov~\cite{Dyckerhoff-Kapranov:1212.3563}.
While Dyckerhoff and Kapranov formulate the condition in 
terms of triangulation of convex polygons, our formulation refers
to the categorical notion of active and inert maps, which
we recall next.

\begin{blanko}{Active and inert maps (generic and free maps).}\label{generic-and-free}
  The category $\simplexcategory$ of nonempty finite ordinals and monotone
  maps has an active-inert factorisation system.  An arrow $a: [m]\to [n]$
  in $\simplexcategory$ is \emph{active} (also called {\em generic}) when
  it preserves end-points, $a(0)=0$ and $a(m)=n$; and it is \emph{inert}
  (also called \emph{free}) if it is distance preserving, $a(i+1)=a(i)+1$
  for $0\leq i\leq m-1$.  A coface map $d^j:[m]\to[m+1]$ is active if and
  only if it is {\it inner}, i.e. $1 \leq j \leq m$.  The active maps are
  generated by the codegeneracy maps and the inner coface maps, while the
  inert maps are generated by the outer coface maps.  Every morphism in
  $\simplexcategory$ factors uniquely as an active map followed by an inert
  map.
  
  The notions of generic and free maps are general notions in category theory,
  introduced by Weber
  \cite{Weber:TAC13,Weber:TAC18}, who extracted the notions from earlier work of
  Joyal~\cite{Joyal:foncteurs-analytiques}; a recommended entry point to the
  theory is Berger--Melli\`es--Weber~\cite{Berger-Mellies-Weber:1101.3064}. We
  have adopted the more recent terminology `active/inert' (due to
  Lurie~\cite{Lurie:HA}), which is more suggestive of the role the two classes
  of maps play.
\end{blanko}

\begin{lemma}\label{genfreepushout}
  Active and inert maps in $\simplexcategory$ admit pushouts along each other, and the 
  resulting maps are again active and inert.
\end{lemma}

\begin{blanko}{Decomposition spaces \cite{GKT:DSIAMI-1}.} 
  A simplicial space $X:\simplexcategory\op\to\Grpd$ is called a {\em decomposition space}
  when it takes active-inert pushouts in $\simplexcategory$ to pullbacks. An example of such a square is \eqref{eq:decompex} above.

  Every Segal space is a decomposition space.
  For example, the nerve of a category or a poset is a decomposition 
  space.
  In a Segal space $X$, all the information is contained in $X_0$ and $X_1$ 
  and the composition map $d_1: X_2 \to X_1$.
  This cannot be said for decomposition spaces in general, but we still have the following
  important property.
\end{blanko}

\begin{lemma}\label{lem:s0d1}
  In a decomposition space $X$, every active face map is a pullback of $d_1:
  X_2 \to X_1$, and every degeneracy map is a pullback of $s_0 :X_0 \to X_1$.
\end{lemma}
\begin{proof}
If we consider the inert maps $f_j:[1]\to[m]$ given by $f_j(0)=j$ and $f_j(1)=j+1$ for $j=0,\dots,m-1$, then we have the following active--inert pushouts in $\simplexcategory$,
 $$\vcenter{\xymatrix{{}[1]\rto^{f_j}\dto_{d^1}&[m]\dto^{d^{j+1}}\\[2]\rto&\ulpullback{}[m+1],}}
\qquad\qquad\vcenter{\xymatrix{{}[1]\rto^{f_j}\dto_{s^0}&[m]\dto^{s^{j}}\\[0]\rto&\ulpullback{}[m-1],}}
$$
which are sent to pullbacks by any decomposition space $X$.
\end{proof}

As far as incidence coalgebras are concerned, 
the notion of decomposition space can be seen as an abstraction of 
that of poset:
it is precisely the condition required to obtain a counital coassociative
comultiplication on $\Grpd_{/X_1}$.  Precisely, the following is the main 
theorem of \cite{GKT:DSIAMI-1}.

\begin{theorem}\label{thm:comultcoass} \cite{GKT:DSIAMI-1}
  For $X$ a decomposition space, the slice $\infty$-category $\Grpd_{/X_1}$ has
  the structure of a strong homotopy comonoid in the symmetric monoidal
  $\infty$-category $\LIN$, with the comultiplication $\Delta$
and counit $\varepsilon$
  defined by the spans
  $$
  X_1 \stackrel{d_1}\longleftarrow X_2 \stackrel{(d_2,d_0)}
  \longrightarrow X_1 \times X_1 ,\qquad \qquad 
  X_1 \stackrel{s_0}\longleftarrow X_0 \longrightarrow 1 . 
  $$
\end{theorem}

If $X$ is the nerve of a locally finite category or poset,
then $X_2$ is the set of composable pairs of 
arrows, and (after passing to $\Q$-vector spaces by taking homotopy cardinality as in \ref{card} and \cite{GKT:HLA}) the formula
is the classical comultiplication formula
$$
\Delta(f) = \sum_{b\circ a=f} a \tensor b .
$$

\begin{blanko}{\culf functors.}
  For the present purposes, the relevant notion of morphism between
  decomposition spaces is that of \culf functors, since these induce
  homomorphisms of the associated incidence coalgebras: a simplicial map
  (between arbitrary simplicial spaces) is called {\em ULF} (\emph{unique
  lifting of factorisations}) if the naturality square for every inner
  coface map is a pullback, and it is called {\em conservative} if the
  naturality square for every codegeneracy map is a pullback.  We write
  {\em \culf} for conservative and ULF, that is, the naturality square for
  every active map in $\simplexcategory$ is a pullback.

  For maps between Rezk complete Segal groupoids, such as fat nerves of 
  categories, the notion of
  conservative is the classical notion, i.e.~only invertible maps are sent to
  invertible ones, and ULF is a homotopy version of the notion of
  unique lifting of factorisations.
\end{blanko}

\begin{blanko}{Example.}\label{ex:graphsdecomp}
  We describe a decomposition space $\GG$ of finite graphs, whose incidence
  coalgebra is the chromatic Hopf algebra of Schmitt~\cite{Schmitt:hacs}.
  This will serve as a running example throughout the paper.  For
  definiteness, by `graph' we will mean simple non-directed graph, though
  other notions of graph would work too.
  
  Let $\GG_n$ be the groupoid of finite graphs with an {\em $n$-layering}
  (meaning an ordered partition of the vertex set into $n$ `layers', which
  may be empty), and isomorphisms between them.  In particular, $\GG_0$ is
  the contractible groupoid consisting only of the empty graph (the only
  graph admitting a $0$-layering), $\GG _1$ is just the groupoid of all
  finite graphs, and $\GG _2$ is the groupoid of finite graphs with vertex
  set partitioned into two.  All the $\GG _n$ assemble into a simplicial
  groupoid: the face maps join two adjacent layers, or project away the
  bottom or top layer; the degeneracy maps insert an empty layer.  It is
  easy to see that this is not a Segal space: a $2$-layered graph cannot be
  reconstructed from the graphs of its layers, since the information about
  edges joining the layers is missing.  One can check that it
  {\em is} a decomposition space: that the square
  $$\xymatrix{
    \GG _2 \ar[d]_{d_0} & \GG _3 \dlpullback \ar[l]_{d_2} \ar[d]^{d_0} \\
    \GG _1 & \GG _2 \ar[l]^{d_1}
  }$$
  is a pullback is to say that a graph with a 3-layering ($\in \GG _3$)
  can be  reconstructed uniquely from a pair of elements in $\GG _2$ with common 
  image in $\GG _1$ (under the indicated face maps). 
  The following picture represents elements corresponding to each other
  in the four groupoids.

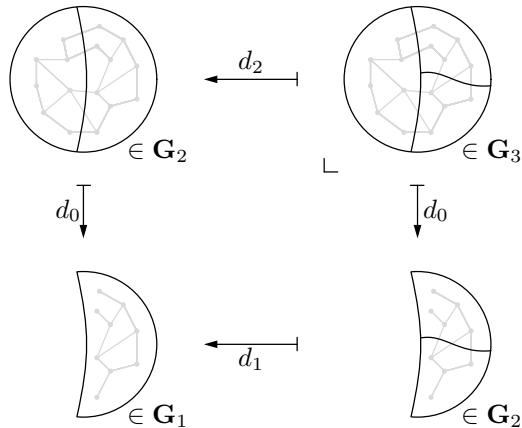
\colorlet{ourgrey}{black!18}

\tikzset{
  greydot/.pic={
	\fill (0,0) circle[radius=0.04];
  }
}

\tikzset{
  biggraph/.pic={
	\begin{scope}[ourgrey, line width=0.5pt]
	  \draw (0.35, 0.263) pic {greydot}
	  -- (0.175, 0.438) pic {greydot}
	  -- (-0.21, 0.263) pic {greydot}
	  -- (-0.263, 0.525) pic {greydot}
	  -- (0.21, 0.7) pic {greydot}
	  -- (0.525, 0.525) pic {greydot}
	  -- (0.7, 0.175) pic {greydot}
	  -- (0.7, -0.263) pic {greydot}
	  -- (0.35, -0.35) pic {greydot}
	  -- (0.175, -0.7) pic {greydot}
	  -- (-0.175, -0.7) pic {greydot}
	  -- (-0.525, -0.438) pic {greydot}
	  -- (-0.613, -0.088) pic {greydot}
	  -- (-0.613, 0.263) pic {greydot}
	  -- (-0.175, -0.14) pic {greydot}
	  -- (0.175, -0.175) pic {greydot};	
	  \draw (0.7, 0.175) -- (0.175, -0.175);	
	  \draw (0.35, -0.35) -- (0.175, -0.175)  -- (0.35, 0.263) -- (0.525, 0.525);	
	  \draw (-0.525, -0.438) -- (-0.175, -0.14) -- (0.175, -0.7);	
	  \draw  (-0.613, 0.263) -- (-0.21, 0.263);
    \end{scope}
  }
}

\tikzset{
  smallgraph/.pic={
	\begin{scope}[ourgrey, line width=0.5pt]
	  \draw (0.35, 0.263) pic {greydot}
	  -- (0.175, 0.438) pic {greydot};
	  \draw (0.21, 0.7) pic {greydot}
	  -- (0.525, 0.525) pic {greydot}
	  -- (0.7, 0.175) pic {greydot}
	  -- (0.7, -0.263) pic {greydot}
	  -- (0.35, -0.35) pic {greydot}
	  -- (0.175, -0.7) pic {greydot};
	  \draw (0.7, 0.175) -- (0.175, -0.175);
	  \draw (0.35, -0.35) -- (0.175, -0.175) pic {greydot}
	  -- (0.35, 0.263) -- (0.525, 0.525);
     \end{scope}
  }
}

\begin{center}
  \vspace*{12pt}
  \begin{tikzpicture}[line width=0.5pt]
	\footnotesize
      
	\begin{scope}[shift={(0.0, 0.0)}]
	  \draw (0.963, -0.963) node {$\in \GG _1$};
	  \draw (0.0, 0.0) pic {smallgraph};      
	  \draw (0.0, 0.0)+(-95:0.963) arc[start angle=-95, end angle=95, radius=0.963];
	  \draw (-0.088, 0.954)  .. controls (0.088, 0.175) and (0.088, -0.175) .. (-0.088, -0.954);
	\end{scope}

	\begin{scope}[shift={(4.375, 0.0)}]
	  \draw (0.963, -0.963) node {$\in \GG _2$};
	  \draw (0.0, 0.0) pic {smallgraph};
	  \draw (0.0, 0.0)+(-95:0.963) arc[start angle=-95, end angle=95, radius=0.963];
	  \draw (-0.088, 0.954) .. controls (0.088, 0.175) and (0.088, -0.175) .. (-0.088, -0.954);
	  \draw (0.044, 0.088) .. controls (0.35, 0.14) and (0.525, -0.123) .. (0.963, -0.088);
	\end{scope}

	\begin{scope}[shift={(0.0, 3.5)}]
	  \draw (0.963, -0.963) node {$\in \GG _2$};
	  \draw (0.0, 0.0) pic {biggraph};
	  \draw (0.0, 0.0) circle (0.963);
	  \draw (-0.088, 0.954) .. controls (0.088, 0.175) and (0.088, -0.175) .. (-0.088, -0.954);
	\end{scope}

	\begin{scope}[shift={(4.375, 3.5)}]
	  \draw (0.963, -0.963) node {$\in \GG _3$};
	  \draw (0.0, 0.0) pic {biggraph};
	  \draw (0.0, 0.0) circle (0.963);
	  	  \draw (-0.088, 0.954) .. controls (0.088, 0.175) and (0.088, -0.175) .. (-0.088, -0.954);
	  \draw (0.044, 0.088) .. controls (0.35, 0.14) and (0.525, -0.123) .. (0.963, -0.088);
	\end{scope}

	\begin{scope}[shift={(3.15, 2.275)}]
	  \draw (0.0, 0.0) -- (0.0, 0.175);
	  \draw (0.0, 0.0) -- (0.175, 0.0);
	\end{scope}
      
	\draw (2.8, 3.413) -- + (0.0, 0.175);
	\draw[->] (2.8, 3.5) -- + (-1.225, 0.0);
	\draw (2.188, 3.71) node {$d_2$};
	\draw (2.8, -0.088) -- + (0.0, 0.175);
	\draw[->] (2.8, 0.0) -- + (-1.225, 0.0);
	\draw (2.188, -0.21) node {$d_1$};
	\draw (-0.088, 2.1) -- + (0.175, 0.0);
	\draw[->] (0.0, 2.1) -- + (0.0, -0.7);
	\draw (-0.21, 1.75) node {$d_0$};
	\draw (4.288, 2.1) -- + (0.175, 0.0);
	\draw[->] (4.375, 2.1) -- + (0.0, -0.7);
	\draw (4.603, 1.75) node {$d_0$};
  \end{tikzpicture}
  \vspace*{12pt}
\end{center}

  The horizontal maps join the last two layers.  The vertical
  maps forget the first layer.  Clearly the diagram commutes.  To reconstruct the
  graph with a 3-layering (upper right-hand corner), most of the
  information is already available in the upper left-hand corner, namely the
  underlying graph and all the subdivisions except the one between layer 2 and
  layer 3.  But this information is precisely available in the lower right-hand
  corner, and their common image in $\GG _1$ says precisely how this missing piece
  of information is to be implanted.

  In the comultiplication formula, $d_1{}\upperstar$ takes a graph
  $G$ to the groupoid of all possible $2$-layerings on $G$, and
  $(d_2,d_0)\lowershriek$ returns the two layers, meaning the
  graphs induced by the two subsets of the vertex set $V$.  After
  taking homotopy cardinality, this is precisely the comultiplication of
  the chromatic Hopf algebra of Schmitt~\cite{Schmitt:hacs}: it takes 
  a basis element $G$ to the sum $\sum G|V_1\otimes G|V_2$, the sum being 
  over all $2$-layerings $V= V_1 + V_2$.

  There is a \culf functor from the decomposition space of graphs
  to the decomposition space of finite sets (defined similarly --- 
  its incidence coalgebra is the binomial coalgebra~\cite{GKT:ex}),
  which to a graph associates its vertex set.  The \culf condition
  simply says that the $n$-layerings on a graph are determined
  by the $n$-layerings of the vertex set.  This \culf functor 
  induces a coalgebra homomorphism from the chromatic coalgebra
  to the binomial coalgebra.  
\end{blanko}

\section{Complete decomposition spaces}

\label{sec:complete}

In this section we introduce the notion of complete 
decomposition spaces, which is needed to talk about 
nondegenerate simplices in a meaningful way.

\begin{blanko}{Complete decomposition spaces.}\label{complete}
  A decomposition space $X$ is called {\em complete} if $s_0:X_0 \to X_1$ is a
  monomorphism of $\infty$-groupoids.
\end{blanko}

\begin{blanko}{Discussion.}\label{complete-discussion}
  It is clear that a Rezk complete Segal space is complete in the sense of
  \ref{complete}.  While it makes sense to state the Rezk completeness
  condition for decomposition spaces too (cf.~\ref{def:Rezk} below), our condition \ref{complete} covers some
  important examples which are not Rezk complete, such as the ordinary nerve of
  a group (cf.\ Example \ref{ex:discr=compl} below). The incidence algebra of the nerve
  of a group is the group algebra --- certainly an example worth covering.

  The completeness condition is necessary to define $\Phieven$ and $\Phiodd$
  (the even
  and odd parts of the `\M functor', see \ref{Phi}) and to establish the \M
  inversion principle at the objective level (Theorem \ref{thm:zetaPhi}).
  The completeness condition is
  also needed to make sense of the notion of length (\ref{length}), and to
  define the length filtration (\ref{def:filt}), which is of independent
  interest, and is also required to be able to take homotopy cardinality of \M inversion.
\end{blanko}

\begin{blanko}{Examples.}\label{ex:discr=compl}
  If a decomposition space $X$ is \emph{discrete}, meaning that each $X_i$ is a set,
  then it will be complete, because $s_0 : X_0 \to X_1$ is a section to $d_0: X_1 \to
  X_0$ and is therefore an injection of sets.  Slightly more generally, a decomposition
  space $X$ will be complete if $d_0:X_1 \to X_0$ is discrete (that is, has discrete
  (homotopy) fibres) since a section to a discrete map is always a monomorphism.
  
  For the simplest example of a decomposition space which is not complete, let
  $G$ be a nontrivial group, and denote the corresponding one-object
  groupoid by $BG$.  Consider the simplicial groupoid $X$ with $X_n = (BG)^n$.
  Here $s_0: 1 \to BG$ is not a monomorphism (although it is a section of
  $BG \to 1$): its (homotopy) fibre is the set of elements of $G$.
\end{blanko}

\begin{blanko}{Example (continued from \ref{ex:graphsdecomp}).}\label{ex:graphscomplete}
  The decomposition space $\GG$ of finite graphs is complete: indeed, $s_0 : \GG _0 \to \GG _1$
  assigns to the empty graph with zero layers the empty graph with one layer.
  Clearly this has trivial automorphism group, so $s_0$ is a monomorphism.
\end{blanko}

The following basic result follows immediately from Lemma~\ref{lem:s0d1}.
\begin{lemma}\label{all-s-mono}
  In a complete decomposition space,
  all degeneracy maps are monomorphisms.
\end{lemma}

\begin{blanko}{Completeness for simplicial spaces.}\label{completesimplicial}
  We shall briefly need completeness also for general simplicial spaces, and the
  first batch of results holds in this generality.  We shall say that $X:\simplexcategory\op\to\Grpd$ is {\em complete} if all degeneracy maps are
  monomorphisms.  In view of Lemma~\ref{all-s-mono}, this agrees with the
  previous definition when $X$ is a decomposition space.
\end{blanko}

\begin{blanko}{Completeness as an exactness condition.}\label{completeexact}
  It is interesting to note that completeness
  is an exactness condition, just as the decomposition-space axiom itself.
  Indeed, for $0 \leq i \leq n$ the squares
  $$\xymatrix{
     [n] \drpullback  & [n]  \ar[l]_{=} \\
     [n] \ar[u]^{=} & [n{+}1] \ar[u]_-{s^i} \ar[l]^{s^i}
  }$$
  are pushouts in $\simplexcategory$, and the completeness condition on  $X:\simplexcategory\op\to\Grpd$ is precisely to send these 
  pushouts to pullbacks (because a map is mono if and only if its pullback 
  along itself is an equivalence).  If $X$ is
  assumed to be a decomposition space, then the completeness condition can be
  expressed by the
requirement that the following single square is a pullback.
  $$\xymatrix{
     X_0 \drpullback \ar[r]^=\ar[d]_= & X_0 \ar[d]^{s_0} \\
     X_0 \ar[r]_{s_0} & X_1 .
  }$$
\end{blanko}

For the rest of this section, $X$ will denote a complete simplicial space,
except where it is explicitly stated to be a complete decomposition space.

\begin{blanko}{Word notation.}\label{w}
  Since $s_0:X_0\to X_1$ is mono,
  we can identify $X_0$ with a full sub-$\infty$-groupoid of $X_1$.  We
  denote by $X_a$ its  complement, the full
  sub-$\infty$-groupoid of {\em nondegenerate $1$-simplices}: 
  $$
  X_1 = X_0 + X_a .
  $$
  We extend this notation as follows.  Consider the alphabet with three letters
  $\{0,1,a\}$.  Here $0$ indicates degenerate edges $s_0(x)\in X_1$, the
  letter $a$ indicates edges which are nondegenerate, and $1$ indicates edges
  which may be degenerate or nondegenerate.
  For $w$ a word in this alphabet $\{0,1,a\}$, of length $|w|=n$,
  put
  $$
  X^w := \prod_{i\in w} X_i \subset (X_1)^n .
  $$
  This inclusion is full since $X_a \subset X_1$ is full by completeness.
  Denote by
  $X_w$ the $\infty$-groupoid of $n$-simplices whose
  principal edges have the types indicated in the word $w$, 
  or more explicitly, the full sub-$\infty$-groupoid of $X_n$ given by the pullback diagram
  \begin{equation}\label{eq:Xw}
\vcenter{  \xymatrix{
  X_w \drpullback \ar[r] \ar[d] & X_n \ar[d] \\
  X^w  \ar[r] &(X_1)^n.
  }}
  \end{equation}
\end{blanko}

\begin{lemma}\label{lem:cons-X1}
  If $X$ and $Y$ are complete simplicial spaces and $f:Y \to X$ is conservative,
  then $Y_a$ maps to $X_a$, and the following square is a
  pullback:
  $$
  \xymatrix{
  Y_1 \ar[d] & Y_a \dlpullback \ar[l] \ar[d]  \\
  X_1 & \ar[l] X_a.}$$
\end{lemma}

\begin{proof}
  This square is the complement of the pullback saying what conservative means.
  But it is general in extensive $\infty$-categories such as $\Grpd$,
  that in the situation
  $$\xymatrix{ A' \ar[r]\ar[d] & A'+B' \ar[d] & \ar[l] B' \ar[d] \\
  A \ar[r] & A+B & \ar[l] B,}
  $$
  one square is a pullback if and only if the other is.
\end{proof}

\begin{cor}\label{cor:cons-wn}
  If $X$ and $Y$ are complete simplicial spaces and $f:Y \to X$ is
  conservative, then for every word $w \in \{0,1,a\}\upperstar$,
  the following square is a pullback:
  \begin{equation}\label{eq:YnYw}
    \vcenter{\xymatrix{
    Y_n\ar[d] & \ar[l] Y_w \dlpullback \ar[d] \\
    X_n & \ar[l]  X_w.}}
  \end{equation}
\end{cor}
\begin{proof}
  The square is connected to 
  \begin{equation}\label{eq:Y1nYw}
    \vcenter{\xymatrix{
    (Y_1)^n\ar[d] & \ar[l] Y^w \dlpullback \ar[d] \\
    (X_1)^n & \ar[l]  X^w}} 
  \end{equation}
  by two instances of pullback-square \eqref{eq:Xw}, one for $Y$ and one for
  $X$.  It follows from Lemma~\ref{lem:cons-X1} that \eqref{eq:Y1nYw} is a pullback,
  hence also \eqref{eq:YnYw} is a pullback, by Lemma~\ref{pbk}.
\end{proof}

\begin{prop}\label{prop:cULF-nondegen}
  If $X$ and $Y$ are complete simplicial spaces and $f:Y \to X$ is \culf, then
  for any word $w\in \{0,1,a\}\upperstar$
  the following square is a pullback:
  $$
  \xymatrix{
    Y_1 \ar[d]_f& \ar[l]  Y_n& \ar[l]  Y_w
    \dlpullback \ar[d]^f \\
    X_1 & \ar[l]   X_n& \ar[l]   X_w
    .}
  $$
\end{prop}
\begin{proof}
  The required pullback square is a horizontal composite 
  $$\xymatrix{
    Y_1 \ar[d]_f& \ar[l]  Y_n\dlpullback\ar[d]^f\dto& \ar[l]  Y_w
    \dlpullback
    \ar[d]^f \\
    X_1 & \ar[l]   X_n& \ar[l]   X_w
    ,}
  $$
  where the right-hand square is the pullback square \eqref{eq:YnYw} of
  Corollary~\ref{cor:cons-wn}.  The horizontal arrows of the left-hand
  square are induced by the unique active map $[n]\to[1]$, and since $f$ is
  \culf this square is a pullback also.
\end{proof}

\begin{lemma}\label{lem:X1w=sum}
  Let $X$ be a complete simplicial space. Then for any words
  $v,v' \in \{0,1,a\}\upperstar$, we have
  $$
  X_{v1v'} = X_{v0v'} + X_{vav'} , 
  $$
  and hence
  $$
  X_n=\sum_{w\in\{0,a\}^n} X_w.
  $$
\end{lemma}
\begin{proof}
  Consider the diagram
  $$
  \xymatrix{
  X_{v0v'} \drpullback  \ar[r]\ar[d] & X _{v1v'} \ar[d] & \ar[l] \dlpullback X_{vav'} \ar[d]\\
  X^{v0v'} \ar[r] & X^{v1v'} & \ar[l] X^{vav'}
  }$$
  The two squares are pullbacks, by an application of Lemma~\ref{pbk},
  since horizontal composition
  of either with the pullback square \eqref{eq:Xw} for $w=v1v'$ gives again the 
  pullback square \eqref{eq:Xw}, for $w=v0v'$ or $w=vav'$.

  Since the bottom row is a sum diagram, it follows that the top row is also
  (since the $\infty$-category of $\infty$-groupoids is extensive).
\end{proof}

We now specialise to complete decomposition spaces, although the following
result will be subsumed in Section~\ref{sec:stiff} on \emph{stiff simplicial spaces}.
\begin{prop}\label{degen-w}
  Let $X$ be a complete decomposition space.
  Then for any words $v,v'$ in the alphabet $\{0,1,a\}$ we have 
  $$
  X_{v0v'}=\Im(s_{|v|}:X_{vv'}\to X_{v1v'}) .
  $$
  That is, the $k$th principal edge of a simplex $\sigma$ is degenerate if and
  only if $\sigma=s_{k-1}d_k\sigma$.
\end{prop}
\noindent 
Recall that $|v|$ denotes the length of the word $v$ and,
as always, the notation $\Im$ refers to the essential image.
\begin{proof}
From \eqref{eq:Xw} we see that (independently of the decomposition-space axiom)
  $X_{v0v'}$ is characterised by the top pullback square in the diagram
  $$\xymatrix{
      X_{v0v'} \drpullback \ar[r]\ar[d] & X_{v1v'} \ar@/^1.5pc/[dd]^{{d_\bot\!}^{|v|}\,{d_\top\!}^{|v'|}} \ar[d]
      \\ 
      X^{v0v'} \drpullback \ar[r]\ar[d] & X^{v1v'} \ar[d]
      \\ 
      X_0 \ar[r]_{s_0} & X_1
  }$$
  But the decomposition-space axiom applied to the exterior pullback diagram
  says that the top horizontal map is $s_{|v|}$, and hence identifies $X_{v0v'}$
  with the image of $s_{|v|}: X_{vv'} \to X_{v1v'}$.  For the final statement,
  note that if $\sigma=s_{k-1}\tau$ then $\tau=d_k\sigma$.
\end{proof}
Combining this with Lemma \ref{lem:X1w=sum} we obtain the following result.
\begin{cor} \label{cor:X1w=sum}
  Let $X$ be a complete decomposition space.  For any words
  $v,v'$ in the alphabet $\{0,1,a\}$ we have
  $$
  X_{v1v'} = s_{|v|}(X_{vv'}) + X_{vav'} .
  $$
\end{cor}

\begin{blanko}{Effective simplices.}\label{effective}
  A simplex in a complete simplicial space $X$ is called {\em effective}
  when all its principal edges are nondegenerate.
  We put
  $$
  \nondeg{X}_n := X_{a\cdots a} \subset X_n ,
  $$
  the full sub-$\infty$-groupoid of $X_n$ spanned by the effective simplices.  (Every
  $0$-simplex is effective by convention: $\nondeg X_0 = X_0$.)  It is clear
  that outer face maps $d_\bot,d_\top:X_{n}\to X_{n-1}$ preserve effective
  simplices, and that every effective simplex is nondegenerate, i.e.~is not in
  the image of any degeneracy map.  It is a useful feature of complete {\em
  decomposition spaces} that the converse is true too:
\end{blanko}

\begin{cor}\label{effective=nondegen}
  In a complete decomposition space $X$, 
  a simplex is effective if and only if it is nondegenerate:
  $$
  \nondeg{X}_n = X_n\setminus {\textstyle\bigcup_{i=0}^n}\Im(s_i).
  $$
\end{cor}
\begin{proof}
  It is clear that $\nondeg{X}_n$ is the complement of 
  $X_{01\cdots1} \cup \cdots \cup X_{1\cdots10}$
  and by Proposition~\ref{degen-w}
  we can identify each of these spaces with the image of a degeneracy map.
\end{proof}

\noindent
In fact this feature is enjoyed 
by a more general class of complete simplicial spaces, called stiff, treated in 
Section~\ref{sec:stiff}.

\medskip

Iterated use of Corollary \ref{cor:X1w=sum} yields 
\begin{cor}\label{cor:Xn}
  For $X$ a complete decomposition space we have 
  $$
  X_{n} = \sum s_{j_k}\dots s_{j_1}(\nondeg X_{n-k}) ,
  $$
  where the sum is over all subsets $\{j_1<\dots<j_k\}$ of $\{0,\dots,n-1\}$.
\end{cor}

\begin{lemma}\label{lem:Segal:nondeg}
  If a complete decomposition space $X$ is a Segal space, then $\nondeg X_n
  \simeq \nondeg X_1 \times_{X_0} \cdots \times_{X_0} \nondeg X_1$, the
  $\infty$-groupoid of strings of $n$ composable nondegenerate arrows in $X_n
  \simeq X_1 \times_{X_0} \cdots \times_{X_0} X_1$.
\end{lemma}
\noindent
This follows immediately from the pullback square \eqref{eq:Xw}.
Note that if furthermore $X$ is Rezk complete, we can say non-invertible instead of 
nondegenerate.

\section{\M inversion in the convolution algebra}

In this section, we establish a M\"obius inversion 
principle at the objective level for arbitrary complete 
decomposition spaces.  (Later we shall impose the
finiteness conditions necessary for taking (homotopy) 
cardinality to obtain the M\"obius inversion principle also at 
the classical `numerical' level.)

\label{sec:Minv}

\begin{blanko}{Convolution.}
  In homotopy linear algebra~\cite{GKT:HLA}, $\infty$-categories
  $\Grpd_{/B}$ play the role of the vector spaces with basis $\pi_0 B$.
  Just as a linear functional is determined by its values on basis
  elements, linear functors $\Grpd_{/B} \to \Grpd$ correspond to arbitrary
  functors $B \to \Grpd$, hence the $\infty$-category $\Grpd^B$ can be
  considered the linear dual of the slice $\infty$-category $\Grpd_{/B}$
  (see~\cite{GKT:HLA} for the precise statements and proof).

If $X$ is a decomposition space, the coalgebra structure on $\Grpd_{/X_1}$
therefore induces an algebra structure on $\Grpd^{X_1}$.
  The convolution product of two linear functors
  $$
  F,G:\Grpd_{/ X_1}\longrightarrow \Grpd,
  $$
  given by spans $ X_1 \leftarrow M \to 1$ and $ X_1 \leftarrow N\to 1$,
  is the composite of their tensor product $F\otimes G$
  and the comultiplication,
  $$
  F*G: \quad \Grpd_{/ X_1}\stackrel{\Delta}\longrightarrow 
  \Grpd_{/ X_1}\otimes \Grpd_{/ X_1} \stackrel{F\tensor G}
  \longrightarrow \Grpd \tensor \Grpd
  \stackrel{\sim}\longrightarrow \Grpd.
  $$
  Thus the convolution product of $F$ and $G$ is given by the composite of spans
  $$
  \xymatrix@!C=9ex{
   X_1 && \\
   X_2 \ar[u] \ar[d]&\ar[l]\ar[d] M *N\ar[lu]\ar[rd]\dlpullback &
  \\
   X_1\times  X_1 &\ar[l]M\times N \ar[r] & 1.
  }$$
  The neutral element for convolution is $\varepsilon:\Grpd_{/X_1}\to\Grpd$ defined
  by the span
  $$
  X_1 \stackrel {s_0}\leftarrow  X_0 \to 1\,.
  $$
\end{blanko}

\begin{blanko}{The zeta functor.}\label{zeta}
  The {\em zeta functor}
  $$
  \zeta:\Grpd_{/ X_1} \to \Grpd
  $$
  is the linear functor defined by the span
  $$
  X_1 \stackrel =\leftarrow  X_1 \to 1\,.
  $$
  We will see later in the locally finite situation (see~\ref{finitary}) that
  on taking the homotopy cardinality of the zeta functor
  one obtains the constant function 1 on $\pi_0 X_1$, that is, the classical zeta
  function in the incidence algebra.
  
  It is clear from the definition of the convolution product that 
  the $k$th convolution power of the zeta functor is given by
  $$
  \zeta^k : \; 
  X_1 \stackrel g\leftarrow  X_k \to 1 ,
  $$
  where $g: [1] \to [k]$ is the unique active map in degree $k$.

Consider also the elements $\delta^a$ and $h^a$ of the incidence algebra given
by the spans
$$
\delta^a :\; X_1 \leftarrow (X_1)_{[a]} \to 1,
\qquad\quad
h^a : \; X_1 \stackrel{\name{a}}\leftarrow 1 \to 1
$$
where $(X_1)_{[a]}$ denotes the component of $X_1$ containing $a\in X_1$.
Then zeta is the sum of the elements $\delta^a$, or the homotopy sum of $h^a$ 
$$
\zeta \;\;\simeq \;\;\sum_{a\in \pi_0 X_1} \delta^a
\;\;\simeq \;\;\int^a h^a
.
$$
\end{blanko}

\begin{blanko}{The idea of \M inversion \`a la Leroux.}
  We are interested in the invertibility of the zeta functor under the
  convolution product.  Unfortunately, at the objective level it can practically
  {\em never} be convolution invertible, because the inverse $\mu$ should always
  be given by an alternating sum (cf.~Theorem~\ref{thm:zetaPhi})
  $$
  \mu = \Phieven - \Phiodd 
  $$
  (of the Phi functors defined below).
  We have no minus sign available, but 
  following the idea of Content--Lemay--Leroux~\cite{Content-Lemay-Leroux},
  developed further by Lawvere--Menni~\cite{LawvereMenniMR2720184},
  we establish the sign-free equations
  $$
  \zeta * \Phieven \simeq \varepsilon + \zeta * \Phiodd ,
  \qquad\qquad
  \Phieven * \zeta  \simeq \varepsilon + \Phiodd * \zeta.
  $$

  In the category case (cf.~\cite{Content-Lemay-Leroux}
  and \cite{LawvereMenniMR2720184}),
  $\Phieven$ (resp.~$\Phiodd$) is given by even-length (resp.~odd-length)
  chains of non-identity arrows.  (We keep the $\Phi$-notation in honour of 
  Content--Lemay--Leroux). In the general setting of decomposition spaces
  we cannot talk about chains of arrows,
  but in the complete case we can still talk about effective simplices
  and their principal edges.
\end{blanko}

From now on we assume again that $X$ is a complete decomposition space.

\begin{blanko}{`Phi' functors.}\label{Phi}
  We define $\Phi_n$ to be the linear functor given by the span
  $$
  X_1 \longleftarrow \nondeg{X}_n \longrightarrow  1,
  $$
 where $\nondeg{X}_n$ is the full sub-$\infty$-groupoid of $X_n$ spanned by the effective simplices, which are the same as the non-degenerate simplices since $X$ is a complete decomposition space, see \ref{effective} and Corollary \ref{effective=nondegen}.
  If $n=0$ then $\nondeg X_0=X_0$ by convention, and $\Phi_0$ is given by the
  span
  $$
  X_1 \longleftarrow X_0 \longrightarrow  1.
  $$
  That is, $\Phi_0$ is the linear functor $\varepsilon$.  Note that
  $\Phi_1=\zeta-\varepsilon$.  The minus sign makes sense here, since $X_0$
  (representing $\varepsilon$) is really a full sub-$\infty$-groupoid of $X_1$ (representing
  $\zeta$).
\end{blanko}

To compute convolution with $\Phi_n$, a key ingredient is the following
general lemma (with reference to the word notation of \ref{w}).

\begin{lemma}\label{lem:X1w}
  Let $X$ be a complete decomposition space. Then for any words
  $v,v'$ in the alphabet $\{0,1,a\}$, the square
  $$\xymatrix{
  X_{vv'} \ar[d]\ar[r] & X_2 \ar[d] \\
  X_v \times X_{v'} \ar[r]& X_1\times X_1
  }$$
  is a pullback.
\end{lemma}
\begin{proof}
  Let $m=|v|$ and $n=|v'|$.
  The square is the outer rectangle in the top row of the diagram
  $$\xymatrix{
  X_{vv'} \ar[d]\ar[r] & X_{m+n} \ar[d] \drpullback \ar[r] 
  & X_{1+n} \ar[d] \drpullback \ar[r]& X_2 \ar[d] 
  \\
  X_v \times X_{v'}\drpullback \ar[d] \ar[r]& X_m\times X_n\ar[d]\ar[r]& X_1\times X_n 
  \ar[r]& X_1\times X_1 
  \\
  X^v \times X^{v'} \ar[r] & {X_1}^m\times{X_1}^n 
  }$$
  The left-hand outer rectangle is a pullback by definition of $X_{vv'}$, and
  the bottom square is a pullback by definition of $X_v$ and $X_{v'}$.  Hence
  the top-left square is a pullback.  But the other squares in the top row are
  pullbacks because $X$ is a decomposition space (compare the square $\xymatrix{*+[o][F-]{1}}$ of \cite[5.3]{GKT:DSIAMI-1}). 
\end{proof}

\begin{lemma}
  We have $$\Phi_n\;\; \simeq\;\; (\Phi_1)^n\;\; =\;\; (\zeta-\varepsilon)^n,$$ 
  the $n$th convolution product of $\Phi_1$ with itself.
\end{lemma}
\begin{proof}
  This follows from the definitions and Lemma~\ref{lem:X1w}.
\end{proof}

\begin{prop}\label{prop:Phi_n}
The linear functors $\Phi_n$ satisfy
$$
\zeta*\Phi_n
\;\;\simeq\;\;
\Phi_n+\Phi_{n+1}
\;\;\simeq\;\;
\Phi_n*\zeta.
$$
\end{prop}
\begin{proof}
  We can compute the convolution $\zeta * \Phi_n$ by
  Lemma~\ref{lem:X1w} as
  $$
  \xymatrix@!C=9ex{
   X_1 && \\
   X_2 \ar[u] \ar[d]&\ar[l]\ar[d] X_{1a\cdots a}\ar[lu]\ar[rd]\dlpullback &
  \\
   X_1\times  X_1 &\ar[l]X_1\times \nondeg X_n \ar[r] & 1
  }$$
  But Lemma~\ref{lem:X1w=sum} tells us that
  $X_{1a\cdots a} = X_{0a\cdots a} + X_{aa\cdots a} = \nondeg X_n +
  \nondeg X_{n+1}$, where the identification in the first summand is
  via $s_0$, in virtue of Proposition~\ref{degen-w}.
  This is an equivalence of $\infty$-groupoids over $X_1$
  so the resulting span is $\Phi_n+\Phi_{n+1}$ as desired.
  The second identity claimed follows similarly.
\end{proof}

Put
$$
\Phieven := \sum_{n \text{ even}} \Phi_n , \qquad
\Phiodd := \sum_{n \text{ odd}} \Phi_n .
$$

\begin{theorem}\label{thm:zetaPhi}
  For a complete decomposition space, the following \M inversion 
  principle holds:
  \begin{align*}
\zeta * \Phieven
 &\;\;\simeq\;\; \varepsilon\;\; +\;\; \zeta * \Phiodd,\\
    \Phieven *\zeta &\;\;\simeq\;\; \varepsilon \;\;+ \;\; \Phiodd*\zeta.
\end{align*}
In fact, these four linear functors are all equivalent.
\end{theorem}

\begin{proof}
  This follows immediately from Proposition~\ref{prop:Phi_n}:
  all four linear functors are equivalent to $\sum_{r\geq0}\Phi_r$.
\end{proof}

We note the following immediate corollary of 
Proposition~\ref{prop:cULF-nondegen}, which can be read as saying
`\M inversion is preserved by \culf functors':
\begin{cor}\label{phi=phi}
  If $f:Y \to X$ is \culf, then $f\upperstar \zeta \simeq \zeta$ and
  $f\upperstar \Phi_n \simeq \Phi_n$ for all $n\geq0$.
\end{cor}

\section{Stiff simplicial spaces}

\label{sec:stiff}

We saw that in a complete decomposition space, degeneracy can be detected on
principal edges.  In Section~\ref{sec:split} we shall come to
split simplicial spaces, which share this property.  A common generalisation is
that of stiff complete simplicial spaces, which we now introduce.

\begin{blanko}{Stiffness.}\label{stiff}
  A simplicial space $X: \simplexcategory\op\to\Grpd$
  is called {\em stiff} if it sends codegeneracy/inert 
  pushouts in $\simplexcategory$ to pullbacks in $\Grpd$.  These pushouts are
  examples of active-inert pushouts, so in particular every
  decomposition space is stiff.
\end{blanko}

\begin{lemma}\label{lem:stiff}
  A simplicial space $X$  is stiff if and only if the following diagrams
  are pullbacks for all $0\leq i\leq n$.
      $$\xymatrix{
	  X_{n} \drpullback 
	  \ar[r]^{s_{i}}\ar[d]&\ar[d]^{{d_\bot\!}^{i}\,{d_\top\!}^{n-i}} X_{n+1}
	  \\ 
	  X_0 \ar[r]_{s_0} & X_1
      }$$
\end{lemma}
\begin{proof}
  The squares in the lemma are special cases of the degeneracy/inert
  squares.  On the other hand, every degeneracy/inert square sits in
  between two of the squares of the lemma in such a way that
  Lemma~\ref{pbk} forces it to be a pullback too.
\end{proof}

The following two lemmas for stiff simplicial spaces are proved in the 
same way as for decomposition
spaces~\cite[Lemmas 3.10 and 3.9 respectively]{GKT:DSIAMI-1}.
\begin{lemma}\label{bonus-pullbacks}
  (`Bonus pullbacks') Let $X$ be a stiff simplicial space.
  For all $n\geq 3$ and all $0<i<j<n$, the following squares of active face and 
  degeneracy maps are pullbacks.
  $$
  \vcenter{\xymatrix{
     X_{n-3} \drpullback \ar[r]^-{s_{i-1}}\ar[d]_{s_{j-2}} & X_{n-2} \ar[d]^{s_{j-1}} \\
     X_{n-2} \ar[r]_-{s_{i-1}} & X_{n-1} 
  }}
  \qquad
  \qquad
  \vcenter{
  \xymatrix{
     X_{n-1} \drpullback \ar[r]^{d_i}\ar[d]_{s_j} & X_{n-2} \ar[d]^{s_{j-1}} \\
     X_{n} \ar[r]_{d_i} & X_{n-1} 
  }}  
  \qquad
  \qquad
  \vcenter{\xymatrix{
     X_{n-1} \drpullback \ar[r]^{d_{j-1}}\ar[d]_{s_{i-1}} & X_{n-2} \ar[d]^{s_{i-1}} \\
     X_{n} \ar[r]_{d_j} & X_{n-1} .
  }}$$
\end{lemma}

\begin{lemma}\label{lem:stiffs0}
  In a stiff simplicial space $X$, every degeneracy map is
  a pullback of $s_0: X_0 \to X_1$.  In particular, if just $s_0:X_0 \to X_1$
  is mono then all degeneracy maps are mono.
\end{lemma}

\begin{cor}\label{lem:conss0}
  A simplicial map $f:Y \to X$ between stiff simplicial spaces is
  conservative if and only if
  the naturality square for $s_0$ is a pullback:
  $$\xymatrix{
     Y_0 \drpullback \ar[r]^{s_0}\ar[d] & Y_1 \ar[d] \\
     X_0 \ar[r]_{s_0} & X_1 .
  }$$
\end{cor}

\begin{cor}
  A stiff simplicial space $X$ is complete if and only if the canonical map from
  the constant simplicial space $X_0$ is conservative.
\end{cor}
\begin{proof}
  This follows from the previous two lemmas and a standard pullback argument,
  exploiting the pullback characterisation of completeness~\ref{completeexact}.
\end{proof}

For complete simplicial spaces, stiffness can be characterised
 in terms of degeneracy:

\begin{prop}\label{halfdecomp}
The following are equivalent for a complete simplicial space $X$.
\begin{enumerate}
\item  $X$ is stiff.
\item  Outer face maps $d_\bot,d_\top:X_{n}\to X_{n-1}$ preserve nondegenerate simplices. 
\item  Any nondegenerate simplex is effective. More precisely,
  $$
  \nondeg{X}_n = X_n\setminus {\textstyle\bigcup_{i=0}^n}\Im(s_{i-1}).
  $$
\item  If the $i$th principal edge of $\sigma\in X_n$ is degenerate, 
then $\sigma=s_{i-1}d_{i-1}\sigma  =s_{i-1}d_{i}\sigma$, that is
  $$X_{1\dots101\dots1} = \Im(s_{i-1}:X_{n-1}\to X_n)$$
\item For each word $w\in\{0,a\}^n$ we have
  $$X_w = \Im(s_{j_k-1}\dots s_{j_1-1}:\nondeg X_{n-k}\to X_n) .
  $$
  where $\{j_1<\dots<j_k\}=\{j:w_j=0\}$.
\end{enumerate}
\end{prop}
\begin{proof}
  $(1)\Rightarrow(2)$: Suppose $\sigma \in X_n$ and that $d_\top \sigma$
is degenerate.  Then $d_\top \sigma$ is in the image
of some $s_i : X_{n-2} \to X_{n-1}$, and hence by
(1) already $\sigma$ is in the image of $s_i : X_{n-1} \to X_n$. 
  
$(2)\Rightarrow(3)$: 
The principal edges of a simplex are obtained by applying outer face maps, 
so nondegenerate simplices are also effective.  For the more precise statement,
just note that both subspaces are full, so are determined by the properties
characterising their objects.

$(3)\Rightarrow(4)$: 
As $\sigma$ is not effective, we have 
$\sigma=s_j\tau$. If $j>i-1$ then the $i$th principal edge of $\sigma$ 
is also that of $\tau$, so by induction $\tau\in\Im(s_{i-1})$. Therefore 
$\sigma\in\Im(s_{i-1})$ also, and 
$\sigma=s_{i-1}d_{i-1}\sigma  =s_{i-1}d_{i}\sigma$ as required. 
If $j< i-1$ the argument is similar.   

$(4)\Leftrightarrow(1)$: 
To show that $X$ is stiff, by Lemma~\ref{lem:stiff} it 
is enough to check that this is a pullback:
      $$\xymatrix{
	  X_{n} \drpullback 
	  \ar[r]^{s_{i}}\ar[d]&\ar[d]^{{d_\bot\!}^{i}\,{d_\top\!}^{n-i}} X_{n+1}
	  \\ 
	  X_0 \ar[r]_{s_0} & X_1
      }$$
But the pullback is by definition $X_{1\cdots 101\cdots 1} \subset X_{n+1}$,
and by assumption this is canonically identified with the image of $s_i : X_n 
\to X_{n+1}$, establishing the required pullback.

$(4)\Leftrightarrow(5)$: This is clear, using Lemma~\ref{lem:X1w=sum}.
\end{proof}

In summary, an important feature of complete stiff simplicial spaces
is that all information about degeneracy is encoded in the principal edges.  We
exploit this to characterise conservative maps between complete stiff simplicial
spaces:
\begin{prop}\label{stiff-cons}
  For $X$ and $Y$ complete stiff simplicial spaces, and $f:Y \to X$ a
  simplicial map, the following are equivalent.
  \begin{enumerate}
    \item $f$ is conservative.  
    \item $f$ preserves the word splitting,
    i.e.~for every word $w\in \{0,a\}\upperstar$, $f$
  sends $Y_w$ to $X_w$.  
    \item $f_1$ maps $Y_a$ to $X_a$.
  \end{enumerate}
\end{prop}
\begin{proof}
  We already saw (Corollary~\ref{cor:cons-wn}) that conservative maps preserve the word 
  splitting (independently of $X$ and $Y$ being stiff), which proves 
  $(1)\Rightarrow(2)$.  The implication $(2)\Rightarrow(3)$ is trivial.
  Finally assume that $f_1$ maps $Y_a$ to $X_a$.
  To check that $f$ is conservative, it is enough (by Corollary~\ref{lem:conss0}) 
  to check that the square
    $$\xymatrix{
     Y_0 \drpullback \ar[r]^{s_0}\ar[d] & Y_1 \ar[d] \\
     X_0 \ar[r]_{s_0} & X_1  
  }$$
  is a pullback.  But since $X$ and $Y$ are complete, this square is just
  $$\xymatrix{
     Y_0 \drpullback \ar[r]^-{s_0}\ar[d] & Y_0 + Y_a \ar[d] \\
     X_0 \ar[r]_-{s_0} & X_0 + X_a  ,
  }$$
  which is clearly a pullback when $f_1$ maps $Y_a$ to $X_a$.
\end{proof}

This proposition can be stated more formally as follows.
For $X$ and $Y$ stiff complete simplicial spaces,
the space of conservative maps $\operatorname{Cons}(Y,X)$ is given
as the pullback
$$\xymatrix@C+2pc{
   \operatorname{Cons}(Y,X) \drpullback \ar[r]\ar@{ >->}[d] & {\displaystyle \prod_{n\in \N} \;
   \prod_{w\in \{0,a\}^n}} \Map(Y_w,X_w) \ar@{ >->}[d] \\
   \operatorname{Nat}(Y,X) \ar[r] & {\displaystyle \prod_{n\in \N}} \;\Map(Y_n,X_n) .
}$$
The vertical arrow on the right is given as follows. We have
$$
\Map(Y_n,X_n)\; \simeq\; \Map(\!\underset{w\in\{0,a\}^n}{\textstyle\sum}\!\! Y_w,\! \underset{v\in\{0,a\}^n}{\textstyle\sum}\!\! X_v )
\;\simeq \!\!\!\prod_{w\in \{0,a\}^n} \!\!\!\Map(Y_w,\!\underset{v\in\{0,a\}^n}{\textstyle\sum} \!\!X_v ).
$$
For fixed $w\in 
\{0,a\}^n$, the space $\Map(Y_w,\sum_{v\in\{0,a\}^n} X_v )$
has a distinguished subobject, namely consisting of those maps that
map into $X_w$ for that same word $w$.

\section{Split decomposition spaces}

\label{sec:split}

In this section, we digress to introduce split 
decomposition spaces, more general than the decomposition 
spaces of locally finite length of the following section.
The interest in this notion is its relation to Kan 
extension of semi-simplicial spaces 
(Theorem~\ref{thm:semisimpl=splitcons}).

\begin{blanko}{Split simplicial spaces.}
  In a complete simplicial space $X$, by definition all
  degeneracy maps are monomorphisms, so in particular it makes sense to
  talk about nondegenerate simplices in degree $n$: these form the full
  sub-$\infty$-groupoid of $X_n$ given as the complement of the degeneracy maps $s_i :
  X_{n-1} \to X_n$.  A simplicial space is \emph{split} if it
  is complete and the face maps preserve nondegenerate simplices.
\end{blanko}

\begin{blanko}{Example (continued).}\label{ex:graphssplit}
  The decomposition space $\GG$ of finite graphs (\ref{ex:graphsdecomp}) is
  split.  Indeed, the face maps join adjacent layers or project away the bottom
  or top layer.  To be nondegenerate means having no empty layers, and this
  property is clearly preserved by the face maps. 
\end{blanko}

By Proposition~\ref{halfdecomp}, a split simplicial space is stiff, 
so the results from the previous section are available for 
split simplicial spaces.  In particular, nondegeneracy can be measured on 
principal edges, and we have
\begin{cor}\label{prop:word=degen}
  If $X$ is a split simplicial space, then the sum splitting
  $$
  X_n\;\; = \sum_{w\in\{0,a\}^n} X_w
  $$
  is realised by the degeneracy maps.
\end{cor}

\begin{blanko}{Non-example.}\label{ex:rsi}
  The strict nerve of any category with a nontrivial section-retraction pair of
  arrows, $r \circ s = \id$, constitutes an example of a complete decomposition
  space which is not split. Indeed, the nondegenerate simplices are the
  chains of composable non-identity arrows, but we have $d_1(s,r)=\id$.
  
  In this way, splitness can be seen as an abstraction of the condition on a
  $1$-category that its identity arrows be indecomposable. We proceed to 
  formalise this, cf.~Corollary~\ref{cor:split=IU} below.
\end{blanko}

\begin{blanko}{Indecomposable units.}\label{IU}
  A simplicial space $X: \simplexcategory\op\to\Grpd$ is said to have 
  {\em indecomposable units} when the following squares are pullbacks for all
  $0 \leq i \leq n$:
  $$\xymatrix{
     X_n \drpullback \ar[r]^-{=}\ar[d]_{s_i s_i} & X_n \ar[d]^{s_i} \\
     X_{n+2} \ar[r]_{d_{i+1}} & X_{n+1} .
  }$$
  
  We note that having indecomposable units is an exactness condition:
  in $\simplexcategory$, the squares
  $$\xymatrix{
     [n] \drpullback & \ar[l]_= [n] \\
     [n{+}2] \ar[u]^{s^i s^i} & [n{+}1] \ar[u]_{s^i} \ar[l]^{d^{i+1}}
  }$$
  are pushouts, and the condition stipulates that they be sent to pullbacks.  
  
  The first instance of the indecomposable-units condition, 
  \begin{equation}\label{eq:IUsq}
    \vcenter{\xymatrix{
     X_0 \drpullback \ar[r]^-{=}\ar[d]_{s_0 s_0} & X_0 \ar[d]^{s_0} \\
     X_{2} \ar[r]_{d_{1}} & X_{1}
  }}\end{equation}
  motivates the name, in view of the following important corollary.
\end{blanko}
\begin{cor}\label{cor:IUtotallydegen}
  For a simplicial space $X$ satisfying the pullback condition~\eqref{eq:IUsq}, 
  if a $2$-simplex $\sigma\in X_2$ 
  has degenerate long edge $d_1 \sigma$ then $\sigma$ itself is totally 
  degenerate.
\end{cor}
\noindent
For the nerve of a category, this is the classical notion of indecomposable
identity arrows.  Note that if $X$ is furthermore complete, then the statement
of the corollary is actually `if and only if'.

\begin{lemma}\label{lem:IU}
  A stiff simplicial space $X$ satisfying the pullback condition~\eqref{eq:IUsq} has indecomposable units.
\end{lemma}
\begin{proof}
  The pullback square for a general instance of the indecomposable-units condition can be connected to the first instance~\eqref{eq:IUsq} by inert face maps,
  and the result follows from stiffness and the usual pullback argument.
\end{proof}

\begin{lemma}\label{lem:IUnondegen}
  For $X$ stiff and complete, we have
  $$
    \vcenter{\xymatrix{
     X_0 \drpullback \ar[r]^-{=}\ar[d]_{s_0 s_0} & X_0 \ar[d]^{s_0} \\
     X_{2} \ar[r]_{d_{1}} & X_{1}
  }}
  \qquad \Leftrightarrow \qquad
      \vcenter{\xymatrix{
     \varnothing \drpullback \ar[r]\ar[d] & X_0 \ar[d]^{s_0} \\
     \nondeg X_{2} \ar[r]_{d_{1}} & X_{1}\,.
  }}
  $$
\end{lemma}

\begin{proof}
  By completeness, we can write
  $X_2 = X_{00} + X_{0a} + X_{a0} + X_{aa}$.  We compute the pullback of
  $s_0$ to each of these summands, exploiting that degenerate principal edges
  only arise from degeneracy maps, cf.~Proposition~\ref{halfdecomp}.
  The first summand gives
  $$\xymatrix{
     X_0 \drpullback \ar[r]^=\ar[d]_= & X_0 \drpullback\ar[rr]^=\ar[d]_{s_0} &  
     & X_0 \ar[d]^{s_0} \\
     X_0 \ar[r]_{s_0} & X_1  \ar[r]_{s_0} & X_2 \ar[r]_{d_1} & X_1 ,
  }$$
  where the left-hand square is a pullback since $X$ is complete.
  The second summand gives
  $$\xymatrix{
     \varnothing \drpullback \ar[r]\ar[d] & X_0 \drpullback\ar[rr]^=
     \ar[d]_{s_0} &  & X_0 \ar[d]^{s_0} \\
     X_a \ar[r] & X_1  \ar[r]_{s_0} & X_2 \ar[r]_{d_1} & X_1 ,
  }$$  
  since $X_a$ and $X_0$ are disjoint in $X_1$.  The third summand is analogous 
  to the second.  In conclusion, the total pullback gives $X_0$ if and only if
  and the fourth summand gives $\varnothing$.  
\end{proof}
   
\begin{prop}
  A simplicial space $X: \simplexcategory\op\to\Grpd$ is split
  if and only if it is stiff, complete and has indecomposable units.
\end{prop}

\begin{proof}
  Suppose $X$ is split.  Then it is complete,
  and it follows from Proposition~\ref{halfdecomp} that
  it is stiff.  By Lemmas~\ref{lem:IU} and \ref{lem:IUnondegen}, it remains 
  just to check that the
  square
$$ \xymatrix{
     \varnothing \drpullback \ar[r]\ar[d] & X_0 \ar[d]^{s_0} \\
     \nondeg X_{2} \ar[r]_{d_{1}} & X_{1}
  }$$
  is a pullback, but this follows from splitness: since $X$ is stiff and complete, 
  nondegenerate is the same as effective (Proposition~\ref{halfdecomp}), so
  splitness implies that $d_1$ 
  maps $\nondeg X_2$ into $\nondeg X_1$, and $\nondeg X_1$ is 
  disjoint from $X_0$.
  
  Suppose now that $X$ is stiff and has indecomposable units.  Fix a simplex
  $\sigma\in X_{n+2}$.  We must show that if $\sigma$ is nondegenerate then also
  $d_j \sigma$ is nondegenerate for all $0 \leq j \leq n+2$.  By stiffness we
  already know that this is the case for $d_j$ inert, so it remains to treat the
  active case.  The contrapositive statement is that if for some $0 < j < n+2$
  we have that $d_j \sigma $ is degenerate then already $\sigma$ is degenerate.
  That is, if we have $d_j \sigma = s_i \tau$ for some indices $0 <j < n+2$ and
  $0\leq i \leq n$, and some simplex $\tau \in X_{n+1}$,
  then there exists a simplex $\rho\in X_{n+1}$ and an index $k$ such that
  $\sigma = s_k \rho$.  There are two cases: if $j=i+1$, then we have the 
  pullback square expressing indecomposable units
  $$\xymatrix{
     X_n \drpullback \ar[r]^-{=}\ar[d]_{s_i s_i} & X_n \ar[d]^{s_i} \\
     X_{n+2} \ar[r]_{d_{i+1}} & X_{n+1} , 
  }$$
  and we can take $\rho=s_i\tau$ and $k=i$.  On the other hand for $j\neq i+1$,
  we have the `bonus pullback' (cf.~Lemma~\ref{bonus-pullbacks})
  $$\xymatrix{
     X_{n+1} \drpullback \ar[r]\ar[d]_{s} & X_n \ar[d]^{s_i} \\
     X_{n+2} \ar[r]_{d_j} & X_{n+1}  ,
  }$$
  and we can take $\rho\in X_{n+1}$ to be the simplex corresponding to 
  $(\sigma,\tau)$ in the pullback.  In either case, we see that $\sigma$ is
  degenerate, as required.
\end{proof}

\begin{cor}\label{cor:splitexact}
  A complete simplicial space $X:\simplexcategory\op\to\Grpd$ is split if and only if
  it preserves pullbacks along degeneracy maps in $\simplexcategory\op$.
  In other words, every degeneracy map forms pullbacks with any other 
  face or degeneracy map.
\end{cor}
\begin{proof}
  Since $X$ is stiff, Lemma~\ref{bonus-pullbacks} says that
  $s_i$ forms pullbacks with all $d_j$ except $d_{i+1}$, but this case is
  covered by having indecomposable units.  On the other hand, again by bonus
  pullbacks, $s_i$ forms pullbacks against all $s_j$ except against itself,
  but this case is covered by being complete (cf.~\ref{completeexact}).
\end{proof}

\begin{cor}\label{cor:split=IU}
  A complete decomposition space is split if and only if it has indecomposable 
  units.
  \qed
\end{cor}

The \emph{long edge} of a simplex $\sigma\in X_n$ in a simplicial
space is the element $g(\sigma)\in X_1$, where $g:X_n\to X_1$ is the unique
active map.

\begin{prop}\label{prop:long-edge-degen}
  In a split simplicial space $X$, if the long edge of a simplex $\sigma\in 
  X_n$ is degenerate
  then the simplex is totally degenerate (that is, in the image of $s_0{}^n$).
\end{prop}
\begin{proof}
  Induction on $n$.  The case $n=2$ is Corollary~\ref{cor:IUtotallydegen}.
  Suppose the proposition is true in dimension $n$ and consider $\sigma'\in
  X_{n+1}$ with long edge $u' := d_1{}^n \sigma'$, assumed degenerate.
  Consider the $2$-simplex
  $\tau := d_1{}^{n-1} \sigma'$ and the $n$-simplex $\sigma := d_{n+1} \sigma'$.
  Then the long edge of $\tau$ is $d_1 \tau = d_1{}^n \sigma' = u'$, and the long
  edge of $\sigma$ is $d_1{}^{n-1}\sigma = d_1{}^{n-1} d_{n+1} \sigma' = d_2
  d_1{}^{n-1} \sigma' = d_2 \tau$:
  \vspace*{15pt}
  $$\xymatrix @! @C=0.4pc @R=0.8pc {
   \sigma':  & & 1 \ar@{}[rrrrd]^(0.25){\txt{$\sigma$}} 
   \ar@/^1.6pc/@{-->}[rrrr]&&&& n \ar[rd]^{d_0 \tau}&  \\
   & 0 \ar[ru]\ar[rrrrru]_(0.6){d_2 \tau} \ar[rrrrrr]_{u'}&&&&&& n{+}1
     \ar@{}[lllu]|(0.4){\txt{$\tau$}}
  }$$
  Since $u'$ is degenerate by assumption,
  $\tau$ is totally degenerate by induction, so in particular its principal
  edges $d_2\tau$ and $d_0\tau$ are degenerate.  But $d_2\tau$ is the long edge
  of $\sigma$, so by induction $\sigma$ is totally degenerate.  Since the
  principal edges of $\sigma'$ are those of $\sigma$ plus $d_0\tau$ in the end,
  we conclude that all principal edges of $\sigma'$ are degenerate, so $\sigma'$
  is totally degenerate by Proposition~\ref{halfdecomp} (as $X$ is stiff).
\end{proof}

\begin{blanko}{Rezk complete simplicial spaces.}\label{def:Rezk}
  A simplicial space $X:\simplexcategory\op\to\Grpd$ is called {\em Rezk complete}
  when $s_0 : X_0 \to X_1^{\operatorname{eq}}$ is an equivalence.  Here 
  $X_1^{\operatorname{eq}}$ is defined as the full sub-$\infty$-groupoid of 
  $X_1$
  spanned by those $f: x \to y$ for which there exists $\sigma,\tau\in X_2$ with
  $d_0 \sigma \simeq f$ and $d_1 \sigma \simeq s_0 y$, and $d_2 \tau \simeq f$ 
  and $d_1 \tau \simeq s_0 x$.  When $X$ is a Segal space, this definition 
  agrees with the usual definition.
\end{blanko}

\begin{lemma}\label{lem:Rezk}
  If a complete simplicial space has indecomposable units then it is Rezk 
  complete.
\end{lemma}
\begin{proof}
  Since $X_1^{\operatorname{eq}} \to X_1$ is mono by construction, and
  $s_0:X_0\to X_1$ is mono by completeness, to show that the two
  sub-$\infty$-groupoids coincide, it is enough to show that every element $f:x
  \to y$ in $X_1^{\operatorname{eq}}$ is actually degenerate.  But if $\sigma\in
  X_2$ exists with $d_1 \sigma \simeq s_0 y$ and $d_0 \sigma \simeq f$, as in
  the definition of $X_1^{\operatorname{eq}}$, then indecomposability of units
  implies that $f$ is degenerate.
\end{proof}

\begin{blanko}{Semi-decomposition spaces.}
  Let $\Deltainj \subset \simplexcategory$ denote the subcategory
  consisting of all the objects and only the injective maps.
  A {\em semi-simplicial} space is an object in the functor $\infty$-category
  $\Fun(\Deltainj\op, \Grpd)$.
  A {\em semi-decomposition} space is a semi-simplicial space preserving
  active-inert pullbacks in $\Deltainj\op$.  Since there
  are no degeneracy maps in $\Deltainj$, this means that we
  are concerned only with pullbacks between active face maps and inert face 
  maps.
  
  Every simplicial space has an underlying semi-simplicial space obtained by
  restriction along $\Deltainj \subset \simplexcategory$.  
  The forgetful functor $\Fun(\simplexcategory\op, \Grpd) \to
  \Fun(\Deltainj\op, \Grpd)$ has a left adjoint given by left Kan 
  extension along $\Deltainj \subset \simplexcategory$:
  $$
  \xymatrix {
  \Deltainj\op \ar[d] \ar[r]^Z & \Grpd \\
  \simplexcategory\op \ar@{..>}[ru]_{\overline Z} &
  }
  $$
  The left Kan extension has the following explicit description:
  \begin{align*}
    \overline Z_0 = & Z_0 \\
    \overline Z_1 = & Z_1 + Z_0\\
    \overline Z_2 = & Z_2 + Z_1 +Z_1 +Z_0\\
    \vdots & \\
    \overline Z_k = & \sum_{w\in\{0,a\}^k} Z_{|w|_a}
  \end{align*}
  For $w\in\{0,a\}^k$ and $\sigma\in Z_{|w|_a}$ the corresponding element
  of $\overline Z_k$ is denoted
$$
s_{i_r}\dots s_{i_2}s_{i_1}\sigma
$$
where $r=k-|w|_a$ and $i_1<i_2<\dots<i_r$ with $w_{i_j}=0$.
  The faces and degeneracies of such elements are defined in the obvious way.
\end{blanko}

\begin{prop}\label{prop:splitKan}
  A simplicial space is split if and only if it is the
  left Kan extension of a semi-simplicial space.
\end{prop}

\begin{proof}
  Given $Z: \Deltainj\op\to\Grpd$, it is clear from the construction
  that the new degeneracy maps in $\overline Z$ are monomorphisms.  Hence 
  $\overline Z$ is complete.
  On the other hand, to say that $\sigma\in \overline Z_n$ is nondegenerate
  is precisely to say that it belongs to the original component $Z_n$, and
  the face maps here are the original face maps, hence map $\sigma$ into
  $Z_{n-1}$ which is precisely the nondegenerate component of $\overline 
  Z_{n-1}$.  Hence $\overline Z$ is split.

  For the other implication, given a split simplicial space $X$, since $X$
  is stiff and complete, we know that nondegenerate is the same as
  effective (Proposition \ref{halfdecomp}) and we have a sum splitting
  $$
  X_n \;\;= \sum_{w\in \{0,a\}^n} X_w .
  $$
  Now by assumption the face maps restrict to the nondegenerate simplices
  to give a semi-simplicial space $\nondeg X : \Deltainj\op \to \Grpd$.
  It is now clear from the explicit description of the left Kan extension
  that $\overline{(\nondeg X_n)} =X_n$, from where it follows readily that
  $X$ is the left Kan extension of $\nondeg X$.
\end{proof}

\begin{prop}
  A simplicial space is a split decomposition space if and only if it is the
  left Kan extension of a semi-decomposition space.
\end{prop}
\begin{proof}
  It is clear that if $X$ is a split decomposition space then $\nondeg X$
  is a semi-decomposition space.  Conversely, if $Z$ is a semi-decomposition
  space, then one can check by inspection that $\overline Z$ satisfies the
  four pullback conditions in ~\cite[Proposition 3.3]{GKT:DSIAMI-1}: two of these
  diagrams concern only face maps, and they are essentially from $Z$, with
  degenerate stuff added.  The two diagrams involving degeneracy maps are
  easily seen to be pullbacks since the degeneracy maps are sum inclusions.
\end{proof}

\begin{blanko}{Example (continued).}\label{ex:graphsKan}
  The split decomposition space $\GG$ of finite graphs (see Examples~\ref{ex:graphsdecomp} and \ref{ex:graphssplit}) 
  is the
  left Kan extension of a semi-simplicial space $Z$ where $Z_n$ is the groupoid of $n$-layered
  graphs with no empty layers.  Here $Z_0$ is still the contractible groupoid
  consisting of the $0$-layered empty graph, and the left Kan extension freely
  adds all the degenerate $n$-layerings for $n>0$.
\end{blanko}

\begin{thm}\label{thm:semisimpl=splitcons}
  The left adjoint functor 
  $\Fun(\Deltainj\op,\Grpd) \to \Fun(\simplexcategory\op\!,\Grpd)$ given by Kan extension along $\Deltainj\subset\simplexcategory$
  induces an equivalence of $\infty$-categories
  $$
\Fun(\Deltainj\op,\Grpd) \simeq \kat{Split}^{\mathrm cons},
$$
the $\infty$-category of split simplicial spaces and conservative maps.
\end{thm}
\begin{proof}
  Let $X$ and $Y$ be split simplicial spaces, then $\nondeg X$ and $\nondeg Y$
  are semi-simplicial spaces whose left Kan extensions are $X$ and $Y$ again.
  The claim is that
  $$
  \operatorname{Cons}(Y,X) \simeq \Nat(\nondeg Y, \nondeg X).
  $$
  Intuitively, the reason this is true can be seen in the first  square as in
  the proof of Lemma~\ref{stiff-cons}: to give a pullback square
    $$\xymatrix{
     Y_0 \drpullback \ar[r]^-{s_0}\ar[d] & Y_0 + Y_a \ar[d] \\
     X_0 \ar[r]_-{s_0} & X_0 + X_a  ,
  }$$
  amounts to giving $Y_0 \to X_0$ and $Y_a\to X_a$ (and of course, in both cases
  this data is required to be natural in face maps), that is to give
  a natural transformation $\nondeg Y 
  \to \nondeg X$.  To formalise this idea, note first that $\Nat(\nondeg Y ,
  \nondeg X)$ can be described as a limit
  $$
  \Nat(\nondeg Y ,\nondeg X)  \longrightarrow \prod_{n\in \N} \Map(\nondeg Y_n, 
  \nondeg X_n) \to \ldots
  $$
  where the rest of the diagram contains vertices indexed by all the face maps,
  expressing naturality.  Similarly
   $\Nat(Y,X)$ is given as a limit
  $$
  \Nat( Y , X)  \longrightarrow \prod_{n\in \N} \Map( Y_n, 
   X_n) \to \ldots
  $$
  where this time the rest of the diagram furthermore
  contains vertices corresponding to degeneracy maps.  The full subspace of 
  conservative maps is given instead as
  $$
  \Cons( Y , X)  \longrightarrow \prod_{w\in \{0,a\}\upperstar} \Map( 
  Y_w, 
   X_w) \to \ldots
  $$
  as explained in connection with Lemma~\ref{stiff-cons}.  Now for each 
  degeneracy map $s_i : X_n \to X_{n+1}$, there is a vertex in the diagram. 
  For ease of notation, let us consider $s_0 : X_n \to 
  X_{n+1}$.  The corresponding vertex sits in the limit diagram as follows:
  for each word $v\in \{0,a\}^n$, we have
  $$\xymatrix{
     {\displaystyle\prod_{w\in \{0,a\}\upperstar} \Map(Y_w,X_w)  }
     \ar[r]^-{\text{proj}}\ar[d]_-{\text{proj}} & \Map(Y_{0v}, X_{0v}) 
     \ar[d]^{\text{pre }s_0} \\
     \Map(Y_v,X_v) \ar[r]_{\text{post }s_0} & \Map(Y_n,X_{n+1}) .     
  }$$
  Now both the pre and post composition maps are monomorphisms with essential
  image $\Map(Y_v, X_{0v})$, so the two projections coincide, which is to say
  that the limit factors through the corresponding diagonal.  Applying this
  argument for every degeneracy map $s_i : X_n \to X_{n+1}$, and for all words,
  we conclude that the limit factors through the product indexed only over the
  words without degeneracies,
  $$
  \prod_{n\in \N} \Map(\nondeg Y_n, \nondeg X_n) .
  $$
  Having thus eliminated all the vertices of the limit diagram that corresponded
  to degeneracy maps, the remaining diagram has precisely the shape of the 
  diagram
  computing $\Nat(\nondeg Y ,\nondeg X)$, and we have already seen that the
  `starting vertex' is the same, $\prod_{n\in \N} \Map(\nondeg Y_n, \nondeg 
  X_n)$.  For the remaining vertices, those corresponding to face maps, 
  it is readily seen that in each case the space is that of the $\Nat(\nondeg 
  Y, \nondeg X)$ diagram, modulo some constant factors that do not play any role in
  the limit calculation.  In conclusion, the diagram calculating $\Cons( Y , X)$
  as a limit is naturally identified with the diagram calculating 
  $\Nat(\nondeg Y, \nondeg X)$ as a limit.
\end{proof}

\begin{prop}\label{prop:split=semi}
  The equivalence of Theorem~\ref{thm:semisimpl=splitcons}
  restricts to an equivalence between semi-decomposition spaces
  and all maps and split decomposition spaces and conservative maps, and it 
  restricts further to an equivalence between semi-decomposition spaces and
  ULF maps and split decomposition spaces and \culf maps.
\end{prop}

\begin{blanko}{Dyckerhoff--Kapranov $2$-Segal semi-simplicial spaces.}
  Dyckerhoff and Kapranov's notion of $2$-Segal space 
  \cite{Dyckerhoff-Kapranov:1212.3563} does not refer to 
  degeneracy maps at all, and can be formulated already for
  semi-simplicial spaces:
  a $2$-Segal space is precisely a simplicial space whose underlying
  semi-simplicial space is a semi-decomposition space.
  We get the following corollary to the results above.
\end{blanko}

\begin{cor}
  Every split decomposition space is the left 
  Kan extension of a $2$-Segal semi-simplicial space.
\end{cor}

\section{The length filtration}
\label{sec:length}

In this section we introduce the notion of length of an
edge ($1$-simplex) in a complete decomposition space, and 
the corresponding notion of locally finite 
length, which endows the resulting coalgebra with an important 
filtration.  Locally finite length is one of two finiteness 
conditions in the notion of M\"obius decomposition space that
we are building up to.

\begin{blanko}{Length.}\label{length}
  Let $a$ be an edge in a complete decomposition space $X$.  The {\em
  length} of $a$ is defined to be the largest dimension of an effective
  simplex (that is, of a nondegenerate simplex, see \ref{effective} and
  Corollary~\ref{effective=nondegen}) with long edge $a$:
  $$
  \ell(a) := \max \{ \dim\sigma \mid \sigma\in \nondeg X, g(\sigma)= a \} ,
  $$
  where as usual $g: X_r \to X_1$ denotes the unique active map.
  More formally: the length is the greatest $r$ such that the pullback 
  $$\xymatrix{
     (\nondeg X_r)_a \drpullback \ar[r]\ar[d] & \nondeg X_r \ar[d]^g \\
     1 \ar[r]_{\name a} & X_1
  }$$
  is nonempty (or $\infty$ if there is no such greatest $r$).
  Length zero can happen only for degenerate edges.
\end{blanko}

\begin{blanko}{Decomposition spaces of locally finite length.}
  A complete decomposition space $X$ is said to have {\em locally finite length}
  when every edge $a\in X_1$
  has finite length.  That is, the pullback
  $$\xymatrix{
     (\nondeg X_r)_a \drpullback \ar[r]\ar[d] & \nondeg X_r \ar[d]^g \\
     1 \ar[r]_{\name a} & X_1
  }$$
  is empty for $r \gg 0$.
  We shall also use the word {\em tight} as synonym for `complete and of locally finite 
  length',
  to avoid confusion with the notion of `locally finite' introduced in 
  Section~\ref{sec:findec}.
\end{blanko}

\begin{eks}
  For posets, the notion of locally finite length coincides with the classical
  notion (see for example Stern~\cite{Stern:1999}), namely that for every $x\leq
  y$, there is an upper bound on the possible lengths of chains from $x$ to $y$.
  When $X$ is the strict (resp.~fat) nerve of a category, locally finite length means that
  for each arrow $a$, there is an upper bound on 
  the length of factorisations of $a$ containing no identity (resp.~invertible)
  arrows.

  A paradigmatic non-example is given by the strict nerve of a category
  containing an idempotent non-identity endo-arrow, $e= e\circ e$: clearly $e$ admits
  arbitrarily long decompositions $e = e\circ \cdots \circ e$.
\end{eks}

\begin{blanko}{Example (continued).}\label{ex:graphstight}
  The decomposition space $\GG$ of finite graphs (Example \ref{ex:graphsdecomp}) is of locally
  finite length, since a graph $G\in \GG_1$ with $k$ vertices can have at most $k$
  nonempty layers.  (For similar reasons, many other examples of decomposition
  spaces of combinatorial nature have locally finite length~\cite{GKT:ex}.)
\end{blanko}

\begin{prop}\label{prop:cULF/FILT=FILT}
  If $f:Y \to X$ is \culf and $X$ is a \FILT decomposition space, then also $Y$ 
  is \FILT.
\end{prop}

\begin{proof}
  Since $X$ is a decomposition space and since $f$ is \culf, also $Y$ is a 
  decomposition space (\cite[Lemma 4.6]{GKT:DSIAMI-1}), and the
  \culf condition ensures that $Y$ is furthermore complete, because the $s_0$ of 
  $Y$ is the pullback of the $s_0$ of $X$.  Finally, $Y$ is also \FILT by
  Proposition~\ref{prop:cULF-nondegen}.
\end{proof}

\begin{prop}\label{prop:tight=>split}
  A tight decomposition space is split.
\end{prop}
\begin{proof}
  A tight decomposition space $X$ is in particular complete and stiff, so
  by Lemmas~\ref{lem:IU} and \ref{lem:IUnondegen} it is enough to prove that for $r=2$ 
  we have a pullback square
  $$\xymatrix{
     \varnothing\drpullback \ar[r]\ar[d] & X_0 \ar[d]^{s_0} \\
     \nondeg X_r \ar[r]_g & X_1 ,
  }$$
  where $g: X_r \to X_1$ is the unique active map (or equivalently, the long-edge map $g$ preserves nondegenerate simplices.)
  We actually prove this for $r\geq 2$.
  Suppose that $\sigma\in \nondeg X_r$ has degenerate long edge $u = g\sigma$.
  The idea is to exploit the decomposition-space axiom to glue together two copies of
  $\sigma$, called $\sigma_1$ and $\sigma_2$, to get a bigger nondegenerate simplex 
  $\sigma_1\#\sigma_2 \in \nondeg X_{r+r}$ again
  with long edge $u$.  By repeating this construction we obtain a contradiction
  to the finite length of $u$.  It is essential for this construction that
  $u$ is degenerate, say $u = s_0 x$, because we glue along the $2$-simplex 
  $\tau = s_0 u = s_1 u = s_0 s_0 x$ which has the property that all three edges are 
  $u$.  Here is a picture of the gluing:
  $$
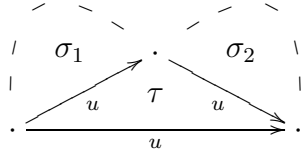
\xymatrix @! @C=1.2pc @R=1.4pc  {
  \\
    && \cdot \ar@/_3pc/@{--}[lld] \ar@/^3pc/@{--}[rrd]  \ar[rrd]_u 
    \ar@{}[ll]|(0.6){\txt{$\sigma_1$}}
    \ar@{}[rr]|(0.6){\txt{$\sigma_2$}}
    \ar@{}[d]|(0.55){\txt{$\tau$}}&& \\
     \cdot \ar[rrrr]_u \ar[rru]_u &&&& \cdot
  }$$
  To formalise this, consider the diagram
  $$\xymatrix @C3.5pc@R3.5pc {
     X_{r+r}\drpullback \ar[r]_{{d_1}^{r-1}}\ar[dd]_{{d_\top}^r}     \ar@/^1.2pc/[rr]^{{d_\bot}^r} & X_{r+1}\drpullback \ar[d]^{{d_2}^{r-1}} 
     \ar@/_1.2pc/[dd]_{{d_\top}^r}
     \ar[r]_{d_\bot} & X_r \ar[d]^{g={d_1}^{r-1}} \\
     & X_2 \ar[d]^{d_\top} \ar[r]_{d_\bot} & X_1 \\
     X_r \ar[r]_{g={d_1}^{r-1}} & X_1 .
  }$$
  The two squares are pullbacks since $X$ is a decomposition space, and the 
  triangles are simplicial identities.  In the
  right-hand square we have $\tau\in X_2$ and $\sigma_2\in X_r$,  with 
  $d_\bot \tau = u = g\sigma_2$.  Hence we get a simplex $\rho\in X_{r+1}$.  This simplex has
  ${d_\top}^r \rho =d_\top\tau= u$, which means that in the left-hand square it matches
  $\sigma_1\in X_r$, to produce altogether the desired simplex
  $\sigma_1\#\sigma_2\in X_{r+r}$.  By construction, this simplex belongs to
  $\nondeg X_{r+r}$: indeed, its first $r$ principal edges are the principal 
  edges of $\sigma_1$, and its last $r$ principal edges are those of $\sigma_2$.
  Its long edge is clearly the long edge of $\tau$, namely $u$
  again, so we have produced a longer decomposition of $u$ than the one
  given by $\sigma$, thus contradicting the finite length of $u$.
\end{proof}

Alternative characterisations of the length of an edge in a
\FILT decomposition space can now be given:
\begin{prop}\label{prop:alt-length}
  Let $X$ be a \FILT decomposition space, and $f\in X_1$.  Then the following
  conditions on $r\in \N$ are equivalent:
  \begin{enumerate}
  \item For all words $w$ in the alphabet $\{0,a\}$
    with $|w|_a\geq r+1$ (that is, the letter $a$ occurs at least $r+1$ times in $w$), the fibre $(X_w)_f$ is empty,
    $$\xymatrix{
       \varnothing \drpullback \ar[r]\ar[d] & X_w \ar[d] \\
       1 \ar[r]_{\name f} & X_1 .
    }$$
    \item For all $k\geq r+1$, the fibre $(\nondeg X_k)_f$ is empty.
    \item The fibre $(\nondeg X_{r+1})_f$ is empty.
  \end{enumerate}
  The length $\ell(f)$ of an edge in a \FILT decomposition space is the least
  $r\in \N$ satisfying these equivalent conditions.
\end{prop}
\begin{proof}
  Clearly $(1)\Rightarrow(2)\Rightarrow(3)$ and, by definition, the length of
  $f$ is the least integer $r$ satisfying $(2)$.  It remains to show that $(3)$
  implies $(1)$.  Suppose (1) is false, that is, we have $w\in\{0,a\}^n$ with
  $k\geq r+1$ occurrences of $a$ and an element $\sigma\in X_w$ with
  $g(\sigma)=f$.  Then by Corollary~\ref{cor:Xn} we know that $\sigma$ is
  an $(n-k)$-fold degeneracy of some $\tau\in\nondeg X_{k}$, and $\sigma$ and
  $\tau$ will have the same long edge $f$.  Finally we see that (3) is false by
  considering the element ${d_1}^{k-r-1}\tau \in X_{r+1}$, which has long edge
  $f$, and is nondegenerate (and hence effective)  since $\tau$ is and
  face maps preserve nondegenerate simplices (as $X$ is split by 
  Proposition~\ref{prop:tight=>split}).
\end{proof}

\begin{blanko}{The length filtration of the space of $1$-simplices.}
  Let $X$ be a \FILT decomposition space.
  We define the $k$th stage of the {\em length filtration} for $1$-simplices to
  consist of all the edges of length at most $k$:
  $$
  X_1^{(k)} := \{ f \in X_1 \mid \ell(f)\leq k\}.
  $$

  Then $X_1^{(k)}$ is
  the full sub-$\infty$-groupoid of $X_1$ given by any of the following equivalent
  definitions:
  \begin{enumerate}
  \item 
  the complement of $\Im( \nondeg X_{k+1} \to X_1)$.
  \item 
  the complement of $\Im( \coprod_{|w|_a>k} X_w \to X_1)$.
  \item
  the full sub-$\infty$-groupoid of $X_1$ whose objects $f$ satisfy 
  $(X_{k+1})_f \subset \bigcup s_i X_k$
  \item
  the full sub-$\infty$-groupoid of $X_1$ whose objects $f$ satisfy 
  $(\nondeg X_{k+1})_f=\varnothing$
  \item
  the full sub-$\infty$-groupoid of $X_1$ whose objects $f$ satisfy 
  $(X_w)_f=\varnothing$ for all $w\in\{0,a\}^{r}$ such that $|w|_a> k$
  \end{enumerate}
\end{blanko}

It is clear from the definition of length that we have a sequence of 
monomorphisms
$$
X_1^{(0)} \into X_1^{(1)} \into X_1^{(2)} \into \dots \into X_1.
$$
The following is now clear.
\begin{prop}
  A complete decomposition space is \FILT if and only if the $X_1^{(k)}$ 
  constitute a filtration, i.e.
  $$
  X_1 = \bigcup_{k=0}^\infty X_1^{(k)} .
  $$
\end{prop}

\begin{blanko}{Length filtration of a \FILT decomposition space.}\label{def:filt}
  Now define the length filtration for all of $X$:
  the length of a simplex $\sigma$ with longest edge $g\sigma=a$ is defined to 
  be the length of $a$:
  $$
  \ell(\sigma) := \ell(a) .
  $$
  In other words, we are defining the filtration in $X_r$ by pulling it back 
  from $X_1$ along the unique active map $X_r \to X_1$.
  This automatically defines the active maps in each filtration degree, 
  yielding an active-map complex
$$
X_\bullet^{(k)} : \Deltagen\op\to\Grpd .
$$

To get the outer face maps, the idea is simply to restrict (since by construction
all the maps $X_1^{(k)}\into X_1^{(k+1)}$ are monos).  We need to check that
an outer face map 
applied to a simplex in $X_n^{(k)}$ again belongs to $X_{n-1}^{(k)}$.
This will be the content of Proposition~\ref{prop:dpresnondeg} below.
Once we have done that, it is clear that we have a sequence of
\culf maps
$$
X_\bullet^{(0)} \into X_\bullet^{(1)} \into \cdots \into X
$$
and we shall see that $X_\bullet^{(0)}$ is the constant simplicial space $X_0$.
\end{blanko}

\begin{prop}\label{prop:dpresnondeg}
  In a tight decomposition space $X$, face maps preserve length: 
  precisely, for any face map $d: X_{n+1} \to X_n$,
  if $\sigma\in X_{n+1}^{(k)}$, then $d\sigma \in X_n^{(k)}$.
\end{prop}

\begin{proof}
  Since the length of a simplex only refers only to its long
  edge, and since an active face map does not
  alter the long edge, it is enough to treat the case of
  outer face maps, and by symmetry it is enough to treat
  the case of $d_\top$.
  Let $f$ denote the long edge of $\sigma$.  
  Let $\tau$ denote the triangle ${d_1}^{n-1}\sigma$.
  It has long 
  edge $f$ again.  Let $u$ and $v$ denote the short edges of $\tau$,
  \begin{equation*}
    \xymatrixrowsep{10pt}
  \xymatrixcolsep{24pt}
\xymatrix @!=0pt {
  & \cdot \ar[rddd]^v & \\
  &&\\
  & \tau &\\
  \cdot \ar[ruuu]^u \ar[rr]_f && \cdot
  }
  \end{equation*}
  that is $v= d_\bot \tau = {d_\bot}^n \sigma$
  and $u= d_\top \tau$, the long edge of $d_\top \sigma$.
  The claim is that if $\ell(f) \leq k$, then $\ell (u)\leq k$.
  If we were in the category case, this would be true since any
  decomposition of $u$ could be turned into a decomposition of $f$
  of at least the same length, simply by postcomposing with $v$.
  In the general case, we have to invoke the decomposition-space
  condition to glue with $\tau$ along $u$.  Precisely, for any 
  simplex $\kappa \in X_w$ with long edge $u$ we can obtain a 
  simplex $\kappa \#_u \tau \in X_{w1}$
  with long edge $f$: since $X$ is a decomposition space,
  we have a pullback square
  $$\xymatrix@C=9pt{
  \kappa \#_u \tau \ar@{}[r]|\in &X_{w1} \drpullback \ar[rr] \ar[d] && X_w \ar[d]^g 
    & \ar@{}[l]|(0.4)\ni \kappa\\
  \tau \ar@{}[r]|\in & X_2 \ar[rr]_{d_\top} \ar[d]_{d_1} && X_1 & 
  \ar@{}[l]|(0.4)\ni u\\
  f \ar@{}[r]|\in &X_1}
  $$
  and $d_\top \tau = u = g(\kappa)$, giving us the desired simplex in
  $X_{w1}$.  With this construction, any simplex $\kappa$ of 
  length $>k$ violating $\ell(u)=k$ 
  (cf.~the characterisation of length given in (1) of 
  Proposition~\ref{prop:alt-length})
  would also yield a simplex $\kappa \#_u \tau$ (of at least the
  same length) violating $\ell(f) =k$.
\end{proof}

The following is an immediate consequence of
Proposition~\ref{prop:long-edge-degen}.
\begin{cor}
  For a tight decomposition space $X$ we have $X_n^{(0)} = X_0$ for all $n$.
  \qed
\end{cor}

\begin{blanko}{Coalgebra filtration.}\label{coalgebrafiltGrpd}
  If $X$ is a \FILT decomposition space, the sequence of
  \culf maps
  $$
  X_\bullet^{(0)} \into X_\bullet^{(1)} \into \cdots \into X 
  $$
  defines coalgebra homomorphisms
  $$
  \Grpd_{/X_1^{(0)}} \to \Grpd_{/X_1^{(1)}} \to \cdots \to \Grpd_{/X_1} 
  $$
  which clearly define a coalgebra filtration of $\Grpd_{/X_1}$.
  
  Recall that a filtered coalgebra is called connected if its $0$-stage
  coalgebra is the trivial coalgebra (the ground ring).  In the present
  situation the $0$-stage is $\Grpd_{/X_1^{(0)}} \simeq \Grpd_{/X_0}$, so we
  see that $\Grpd_{/X_1}$ is connected if and only if $X_0$ is contractible.
  
  On the other hand, the $0$-stage elements are precisely the 
  degenerate edges,
  which almost tautologically are group-like.  Hence the incidence coalgebra of
  a \FILT decomposition space will always have the property that the
  $0$-stage is spanned by group-like elements.  For some purposes, this property
  is nearly as good as being connected (cf.~\cite{Kock:1411.3098}, 
  \cite{Kock:1512.03027} for this 
  viewpoint in the context of renormalisation).
\end{blanko}

\begin{blanko}{Grading.}\label{grading}
  Given a $2$-simplex $\sigma\in X_2$ in a complete decomposition space $X$,
  it is clear that we have
  $$
  \ell(d_2\sigma) + \ell(d_0 \sigma) \leq \ell(d_1\sigma)
  $$
  generalising the case of a category, where $f=ab$  implies 
  $\ell(a)+\ell(b)\leq\ell(f)$.
  In particular, the following configuration of edges 
  illustrates that one does not in general have equality:
  $$
  
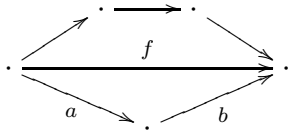
\xymatrix @! @R=10pt @C=5pt {
  && \cdot \ar[rr] && \cdot \ar[rrd] && \\
  \cdot \ar[rru] \ar[rrrrrr]^f \ar[rrrd]_a &&&&&& \cdot \\
  &&& \cdot \ar[rrru]_b &&&
  }
  $$
  Provided none of the edges can be decomposed further, we have
  $\ell(f) = 3$, but $\ell(a) = \ell(b) =1$.
  For the same reason, the length filtration is not in general 
  a grading:  $\Delta(f)$ contains the term $a\tensor b$
  of degree splitting $1+1 < 3$.
  Nevertheless, it is actually common in examples of interest
  to have a grading: this happens when all maximal chains composing to a given 
  edge $f$ have the same length, $\ell(f)$.  
  Many examples from combinatorics have this property~\cite{GKT:ex}.
  
  The abstract formulation of the condition for the length filtration to
  be a grading is this:
  For every $k$-simplex $\sigma \in X_k$ with long edge $a$ and principal edges
  $e_1,\ldots,e_k$, we have
  $$
  \ell (a) = \ell(e_1) + \cdots + \ell(e_k) .
  $$
  Equivalently, for every $2$-simplex $\sigma\in X_2$
  with long edge $a$ and short edges
  $e_1,e_2$, we have
  $$
  \ell (a) = \ell(e_1) +\ell(e_2) .
  $$
  
  The length filtration is a grading if and only if the functor $\ell :X_1 \to 
  \N$ extends to a simplicial map to the nerve of the monoid $(\N,+)$ (this map is
  rarely \culf, though).  
  
  If $X$ is the nerve of a poset $P$, then the length filtration is a grading
  if and only if $P$ is {\em ranked}, i.e.~for any $x,y \in P$, every maximal chain
  from $x$ to $y$ has the same length~\cite{Stanley}.
\end{blanko}

\begin{blanko}{Example (continued).}
  The decomposition space $\GG$ of finite graphs (Example \ref{ex:graphsdecomp}) is graded
  by the number of vertices.
\end{blanko}

\section{Locally finite decomposition spaces}
\label{sec:findec}

In order to be able to take homotopy cardinality of the $\Grpd$-coalgebra obtained from a
decomposition space $X$ to get a coalgebra at the numerical level (vector
spaces), we need to impose certain finiteness conditions on $X$.  Firstly, just
for the coalgebra structure to have a homotopy cardinality, we need $X$ to be {\em
locally finite} (\ref{finitary}) but it is not necessary that $X$ be complete.
Secondly, in order for \M inversion to admit a homotopy cardinality,
what we need in addition is
precisely the filtration condition (which in turn assumes completeness).  We
shall define a {\em \M decomposition space} to be a locally finite \FILT
decomposition space (\ref{M}).

We begin with a few reminders on finiteness of $\infty$-groupoids.

\begin{blanko}{Finiteness conditions for $\infty$-groupoids.}\label{bl:finiteness}
  (Cf.~\cite{GKT:HLA}.)
  An $\infty$-groupoid $B$ is {\em locally finite} if at each base point $b$ the
  homotopy groups $\pi_i (B,b)$ are finite for $i\geq1$ and are trivial for $i$
  sufficiently large.  It is called {\em finite} if furthermore it has only
  finitely many components.
  We denote by $\grpd$ the $\infty$-category of finite $\infty$-groupoids.
 A map of $\infty$-groupoids is {\em finite} if its fibres are~\cite[\S3]{GKT:HLA}.
  The role of vector spaces is played by finite-$\infty$-groupoid slices $\grpd_{/B}$
  (where $B$ is a locally finite $\infty$-groupoid), while the role of 
  profinite-dimensional vector spaces is played by finite-presheaf $\infty$-categories
  $\grpd^B$.  Linear maps are given by spans of {\em finite type}, meaning $A \stackrel 
  p\leftarrow M \stackrel q\to B$ in which $p$ is a finite map.  Prolinear maps are given
  by spans of {\em profinite type}, where $q$ is a finite map.  
  Inside the $\infty$-category $\LIN$, we have
  two $\infty$-categories: $\ind\lin$ whose objects are the 
  finite-$\infty$-groupoid slices $\grpd_{/B}$ and whose mapping spaces are $\infty$-groupoids of 
  finite-type spans, and the $\infty$-category $\pro\lin$ whose objects are 
  finite-presheaf $\infty$-categories $\grpd^B$, and whose mapping spaces are $\infty$-groupoids of 
  profinite-type spans.
  
  We shall also need $\Grpd^{\relfin}_{/B}$, the full subcategory of
  $\Grpd_{/B}$ spanned by the finite maps $p:X\to B$, and $\grpd^B_{\finsup}$,
  the full subcategory of $\Grpd^B$ spanned by presheaves with finite values and
  finite support.  By the support of a presheaf $F: B \to \Grpd$ we mean the
  full sub-$\infty$-groupoid of $B$ spanned by the objects $b$ for which
  $F(b)\neq\varnothing$.
\end{blanko}

\begin{prop}[{Cf.~\cite[Proposition~4.3]{GKT:HLA}}]\label{finitetypespan}
  For a span $A \stackrel p\leftarrow M \stackrel q\to B$ of locally finite 
  $\infty$-groupoids,
the following are equivalent.
\begin{enumerate}
\item $p$ is finite.
\item The linear functor $F:= q\lowershriek \circ p\upperstar :
\Grpd_{/A}\to \Grpd_{/B}$ restricts to
$$
\grpd_{/A}\stackrel{p\upperstar}\longrightarrow 
\grpd_{/M}\stackrel{q\lowershriek}\longrightarrow \grpd_{/B}.
$$
\item The transpose $F^t := p\lowershriek\circ q\upperstar
: \Grpd_{/B} \to \Grpd_{/A}$
restricts to
$$
\Grpd^{\relfin}_{/B}
\stackrel{q\upperstar}\longrightarrow 
\Grpd^{\relfin}_{/M}\stackrel{p\lowershriek}\longrightarrow 
\Grpd^{\relfin}_{/A}.
$$
\item The dual functor $F^\vee: \Grpd^B \to \Grpd^A$ restricts to 
$$
\grpd^B
\to
\grpd^A.
$$
\item The dual of the transpose, $F^t{}^\vee :\Grpd^A \to \Grpd^B$ restricts to
$$
\grpd^A_{\finsup} 
\to
\grpd^B_{\finsup}.
$$
\end{enumerate}
\end{prop}

\begin{blanko}{(Homotopy) cardinality.}
  (Cf.~\cite{Baez-Dolan:finset-feynman,GKT:HLA})\label{card}
  The homotopy cardinality of a finite $\infty$-groupoid $B$ is
by definition
$$
\norm{B} := \sum_{b\in \pi_0 B} \prod_{i>0} \norm{\pi_i(B,b)}^{(-1)^i} .
$$
Here the norm signs on the right refer to order of homotopy groups.
From now on we will just say {\em cardinality} for homotopy 
cardinality.

For each locally finite $\infty$-groupoid $B$, there is a `relative' notion of cardinality
$$\norm{ \;\, } : \grpd_{/B} \longrightarrow \Q_{\pi_0 B},$$ sending 
a basis element $\name b$
to the basis element $\norm{\name b} := \delta_b$ corresponding 
to $b\in \pi_0 B$.
The delta notation for these basis elements is useful to keep track of
the level of discourse.

Dually, there is a notion of cardinality
$\norm { \;\, } : \grpd^B \to \Q^{\pi_0 B}$.
The profinite-dimensional vector space $\Q^{\pi_0 B}$
is spanned by the characteristic functions $\delta^b
=
\frac{\norm{h^b}}{\norm{\Omega(B,b)}}$, 
the cardinality of the representable functor $h^b$ divided by the
cardinality of the loop space.
\end{blanko}

\begin{blanko}{Locally finite decomposition spaces.}\label{finitary}
  A decomposition space $X:\simplexcategory\op \to \Grpd$ is called
  \emph{locally finite} if $X_1$ is locally finite and both 
  $s_0:X_0 \to X_1$ and $d_1:X_2 \to X_1$ are finite maps.
\end{blanko}

\begin{lemma}
  Let $X$ be a decomposition space.
  \begin{enumerate}
    \item If $s_0:X_0\to X_1$ is finite then so are all degeneracy maps $s_i:X_n\to X_{n+1}$.
    \item If $d_1:X_2\to X_1$  is finite then so are all active face maps 
          $d_j:X_{n}\to X_{n-1}$, $j\neq0,n$.
    \item  $X$ is locally finite if and only if $X_n$ is locally finite for every
      $n$ and $g:X_m\to X_n$ is finite for every active map $g:[n]\to[m]$ in $\simplexcategory$.
  \end{enumerate}
\end{lemma}

\begin{proof}
  Since finite maps are stable under pullback \cite[Lemma 3.13]{GKT:HLA},
  both (1) and (2) follow from Lemma~\ref{lem:s0d1}.
  
  Re (3): If $X$ is locally finite, then by definition $X_1$ is locally finite, and 
  for each $n\in \N$ the unique active map $X_n \to X_1$ is finite by (1) or 
  (2).  It follows that $X_n$ is locally finite \cite[Lemma~3.15]{GKT:HLA}.
  The converse implication is trivial.
\end{proof}

\begin{blanko}{Remark.}
  If $X$ is the nerve of a poset $P$, then it is locally finite in the above
  sense if and only if it is locally finite in the usual sense of 
  posets~\cite{Stanley},
  viz.~for every $x, y\in P$, the interval $[x,y]$ is finite.  The points in
  this interval parametrise precisely the two-step factorisations of the unique 
  arrow $x\to y$, so this condition amounts to $X_2 \to X_1$ having finite 
  fibre over $x \to y$.  (The condition $X_1$ locally finite is void in this 
  case, as any discrete set is locally finite;  the condition on $s_0 : X_0 \to 
  X_1$ is also void in this case, as it is always just an inclusion.)
  
  For posets,
  `locally finite' implies `locally finite length'.  (The converse is not true:
  take an infinite set, considered as a discrete poset, and adjoin a top and a 
  bottom element: the result is of locally finite length but not locally finite.)
  Already for categories, it is not true that locally finite implies locally
  finite length: for example the strict nerve of a finite group is locally
  finite but not of locally finite length.
\end{blanko}

\begin{blanko}{Example (continued).}\label{ex:graphslocfin}
  The decomposition space $\GG$ of finite graphs~(Example~\ref{ex:graphsdecomp}) is locally
  finite.  Indeed, $\GG_1$ is locally finite since a finite graph has finite
  automorphism group; the map $s_0: \GG_0 \to \GG_1$ is finite since it is a
  monomorphism (see Example~\ref{ex:graphscomplete}), and $d_1: \GG_2 \to \GG_1$ is finite
  since a given graph admits only finitely many different $2$-layerings.  (For
  similar reasons, many other examples of decomposition spaces of combinatorial
  nature are locally finite~\cite{GKT:ex}.)
\end{blanko}

\begin{blanko}{Numerical incidence algebra.}
  It follows from Proposition~\ref{finitetypespan} that, for any locally finite decomposition space
  $X$, the comultiplication maps
  \begin{eqnarray*}
    \Delta_n: \Grpd_{/ X_1} & \longrightarrow & 
    \Grpd_{/ X_1\times X_1\times \dots\times X_1  } 
  \end{eqnarray*}
  given for  $n\geq 0$ by the spans
  $$
  \xymatrix@R-15pt{
   X_1  & \ar[l]-_{m}  X_n\ar[r]^-{p} &  X_1\times X_1\times \dots\times X_1  
  }
  $$
  restrict to linear functors
  \begin{eqnarray*}
    \Delta_n: \grpd_{/ X_1} & \longrightarrow & 
    \grpd_{/ X_1\times X_1\times \dots\times X_1  } .
  \end{eqnarray*}
  The linear functors $\Delta_2$ and $\Delta_0$ are just the comultiplication $\Delta$ and the counit $\varepsilon$ of Theorem \ref{thm:comultcoass},
$$
\grpd_{/X_1}
\stackrel{\Delta}\rTo 
\grpd_{/X_1\times X_1}
,\qquad\qquad\qquad
\grpd_{/X_1}
\stackrel{\varepsilon}\rTo
\grpd.
$$
and we can take their cardinality to obtain a coalgebra structure,
$$
\Q_{\pi_0 X_1}
\stackrel{\norm{\Delta}}\rTo \Q_{\pi_0 X_1}\tensor\Q_{\pi_0 X_1} 
,\qquad\qquad
\Q_{\pi_0 X_1}
\stackrel{\norm{\varepsilon}}
\rTo
\Q\qquad
$$
termed the \emph{numerical incidence coalgebra} of $X$.
\end{blanko}

\begin{blanko}{Morphisms.}\label{rmk:morphisms}
  It is worth noticing that for {\em any} \culf functor $F: Y \to X$ between
  locally finite decomposition spaces,
  the induced coalgebra homomorphism $F\lowershriek: \Grpd_{/Y_1}\to
  \Grpd_{/X_1}$ restricts to a functor $\grpd_{/Y_1} \to \grpd_{/X_1}$.  In other
  words, there are no further finiteness conditions to impose on \culf 
  functors.
\end{blanko}

\begin{blanko}{Numerical convolution product.}
  By duality, if $X$ is locally finite, the convolution
  product descends to the profinite-dimensional vector space $\Q^{\pi_0 X_1}$
  obtained by taking cardinality of $\grpd^{X_1}$.  It follows from the general
  theory of homotopy linear algebra (see \cite{GKT:HLA}) that the cardinality of
  the convolution product is the linear dual of the cardinality of the 
  comultiplication.  Since it is the same span that defines the comultiplication
  and the convolution product, it is also the exact same matrix
  that defines the cardinalities of these two maps.
  It follows that the structure constants for the convolution product (with 
  respect to the pro-basis $\{\delta^x\}$) are the same as the structure 
  constants for the comultiplication (with respect to the basis
  $\{\delta_x\}$).
  These are classically called the section coefficients, and we proceed to
  derive formulae for them in simple cases.
\end{blanko}

\label{sec:tioncoeff}
Let $X$ be a locally finite decomposition space.
The comultiplication at the objective level
\begin{eqnarray*}
  \grpd_{/X_1} & \longrightarrow & \grpd_{/X_1\times X_1}  \\
  \name f & \longmapsto & \big[ R_f : (X_2)_f \to X_2 \to X_1 \times X_1 \big]
\end{eqnarray*}
yields a comultiplication of vector spaces by
taking cardinality (remembering that $\norm{\name f} = \delta_f$):
\begin{eqnarray*}
\Q_{\pi_0 X_1} & \longrightarrow & \Q_{\pi_0 X_1} \tensor \Q_{\pi_0 X_1} \\
\delta_f & \longmapsto & \norm{R_f} \\
&=& 
\int^{(a,b)\in X_1\times X_1} \norm{(X_2)_{f,a,b}} \delta_a \tensor \delta_b \\
&=& \sum_{a,b}  \norm{(X_1)_{[a]}}\norm{(X_1)_{[b]}} \norm{(X_2)_{f,a,b}} 
\delta_a \tensor \delta_b.
\end{eqnarray*}
where $(X_2)_{f,a,b}$ is the fibre over the three face maps.
The integral sign is a sum weighted by homotopy groups.
These weights together with the cardinality of the triple fibre are called
the {\em section coefficients}, denoted
$$
c^f_{a,b} := \norm{(X_2)_{f,a,b}} \cdot 
\norm{(X_1)_{[a]}}\norm{(X_1)_{[b]}} ,
$$
that is, the numbers $c^f_{a,b}$ such that
$$
\Delta(\delta_f) = \sum_{a,b} c^f_{a,b} \; \delta_a \otimes \delta_b ,
$$
in the numerical incidence coalgebra of $X$.

\bigskip

In the case where $X$ is a Segal space (and even more, when $X_0$ is a 
$1$-groupoid), we can be very explicit about the section coefficients.

\begin{lemma}\label{lem:Omega(y)}
\ourmargin{Slightly rewritten}
  For $X$ a Segal space, the fibre of $X_2\to X_1\times X_1$ over $(a,b)$
  is naturally identified with $\Map_{X_0}(d_0 a,d_1 b)$. Precisely, we 
  have a pullback square
  $$
  \xymatrix{
  \Map_{X_0}(d_0 a,d_1 b) \drpullback \ar[r] \ar[d] & X_2 \ar[d]^-{(d_2,d_0)} \\
  1 \ar[r]_-{\name{(a,b)}}& X_1 \times X_1 .
  }
  $$
\end{lemma}
\begin{proof}
  We can compute the pullback as
  $$
  \xymatrix{
  \Map_{X_0}(d_0 a,d_1 b) \drpullback \ar[r] \ar[d] & X_2 \drpullback\ar[d]_{(d_2,d_0)}  
  \ar[r]^-{d_0\circ d_2} & X_0 \ar[d]^{\text{diag}}\\
1 \ar[r]_-{\name{(a, b)}}
  & X_1 \times X_1 \ar[r]_-{d_0\times d_1} & X_0\times X_0 .
  }
  $$
  Indeed, the right-hand square is a pullback because $X$ is a Segal space
  (using the general fact that a square
  $$
  \xymatrix{
  A \ar[r]^g \ar[d]_f & C \ar[d]^k \\
  B \ar[r]_h& D
  }
  $$
  is a pullback if and only if
  $$
  \xymatrix{
  A \ar[r]^{k\circ g} \ar[d]_{(f,g)} & D \ar[d]^-{\text{diag}} \\
  B \times C \ar[r]_-{h\times k}& D\times D
  }
  $$
  is a pullback). Now the result follows since the fibre of the diagonal is
  the mapping space, as required.
\end{proof}

Note that we can express this in terms of loop spaces as
  $$
  \Map_{X_0}(d_0 a,d_1 b)
  \ \simeq \
  \begin{cases}
   \Omega(X_0,y) & \text{ if } d_0 a \simeq y \simeq d_1 b \\
   \varnothing
   & \text { else},
  \end{cases}
$$
but this is a non-canonical identification, and in the following it is
preferrable to stick with the mapping space.

\bigskip

Now suppose furthermore that $X_0$ is a $1$-groupoid,
so that the diagonal map $X_0\to X_0\times X_0$ is discrete,
and therefore the fibre $\Map_{X_0}(d_0 a,d_1 b)$ 
in Lemma~\ref{lem:Omega(y)} is discrete.
To find the section coefficient $c_{a,b}^f$, we
calculate the fibre
\begin{equation}\label{X2fab}
\xymatrix{
(X_2)_{f,a,b} \ar[r] \ar[d] \drpullback & \Map_{X_0}(d_0 a,d_1 b) \ar^{\ell}[d]
\\
1 \ar_{\name{f}}[r] & X_1 ,
}\end{equation}
over
$f\in X_1$ of the map
$$  \xymatrix{
\ell: \ \Map_{X_0}(d_0 a,d_1 b) \ar[r]  
& X_2\rto^-{d_1}&X_1}.
$$
(Note that if $X$ is the fat nerve of a category, so that $X_2 \cong X_1 
\times_{X_0} X_1$ and we have well-defined arrow composition, then $\ell(\varphi)$ can be
taken to be the composite $a \varphi b$ in the category.)

The important point is that $\Map_{X_0}(d_0 a,d_1 b)$ is discrete.
We can now apply the following general lemma.

\medskip

\noindent
{\bf Lemma.}
{\em For $D$ a discrete space, and $D_b$ defined as the fibre
$$\xymatrix{
D_b \ar[r] \ar[d] \drpullback & D \ar^{\ell}[d]
\\
1 \ar_{\name{b}}[r] & B 
}$$
we  have
$$
D_b \simeq \sum_{d\in D} \Map_B(b,\ell(d)) .
$$
}

\begin{proof}
   We compute $D_b$ as the sum of the fibres for 
the upper square in the diagram
$$\xymatrix{
D_{b,d} \drpullback \ar[d] \ar[r] & 1 \ar^{\name{d}}[d] \\
D_b \ar[r] \ar[d] \drpullback & D \ar^{\ell}[d]
\\
1 \ar_{\name{b}}[r] & B .
}$$
Since $D$ is discrete, we get
$$
D_b \simeq \sum_{d\in D} D_{b,d} \simeq \sum_{d\in D} \Map_B(b,\ell(d)) .
$$

\vspace*{-16pt}

\end{proof}

\medskip

Applying this to the situation of \eqref{X2fab} now yields

\begin{cor}\ourmargin{Corrected}
  Suppose $X$ is a Segal space, and that $X_0$ is a $1$-groupoid. 
  Given $a,b,f\in X_1$ 
  we have
  $$
  (X_2)_{f,a,b} \ \simeq \ \sum_{d_0 a\stackrel\varphi\to d_1 b} 
  \Map_{X_1}(f,\ell(\varphi))
  $$
\end{cor}

The mapping space depends only on whether $f$ and $\ell(\varphi)$ are
equivalent or not, in the space $X_1$. Precisely:
$$
\Map_{X_1}(f,\ell(\varphi)) \simeq \begin{cases} \Omega(X_1,f) & \text{ if 
$f\simeq \ell(\varphi)$} \\ \emptyset & \text{ else} \end{cases}
$$
So we need as many copies of the loop space as there are $\varphi$  giving 
$f$.  Taking cardinality, we therefore arrive at the following 
formula for the section coefficients.
 Recall that in the case of the fat 
nerve of a category, $\ell(\varphi) = a \varphi b$.

\begin{cor}\hspace*{-2mm}\footnote{We are grateful to Alex Cebrian for pointing out a
mistake in the published version of these two corollaries. He
also suggested the present corrected version together with a
proof, whose main argument he had already used in \cite{Cebrian}.}\label{sectioncoeff!}
  \ourmargin{Corrected}
  Suppose $X$ is the fat nerve of a category. Then the section coefficients
  of the incidence coalgebra
  are given by
  $$
  c^f_{a,b} = \frac{\norm{\Aut(f)} \cdot \norm{\{ \varphi \mid a \varphi b \simeq f \}
}}{ \norm{\Aut(a)} \cdot 
  \norm{\Aut(b)}} .
  $$
Here $\{ \varphi \mid a \varphi b \simeq f \}$ is the {\em set} 
consisting of those $\varphi : d_0 a \isopil 
  d_1 b$ for which there exists an isomorphism $a\varphi b \simeq f$.
\end{cor}

\noindent
{\bf Example.} \ourmargin{New}
Consider the fat nerve of the category of finite sets and surjections,
whose incidence bialgebra is the Fa\`a di Bruno 
bialgebra~\cite{GalvezCarrillo-Kock-Tonks:1207.6404,GKT:ex}.
Consider $f: \un 5 \to \un 2$, \
$a:\un 5\to \un 3$, and $b:\un 3\to \un 2$ given as

%
%
%
%
%
%
%

\begin{center}
  \begin{tikzpicture}[line width=0.25mm]
	\draw (0.0, 1.575);
	\draw (0.0, -0.525);
	\begin{scope}[shift={(0.0, 0.0)}]
	  \fill (0.0, 1.4) circle[radius=0.065];
	  \draw (0.0, 1.4) -- (1.05, 0.875);
	  \fill (1.05, 0.875) circle[radius=0.065];
	  \fill (0.0, 1.05) circle[radius=0.065];
	  \draw (0.0, 1.05) -- (1.05, 0.875); 
	  \fill (0.0, 0.7) circle[radius=0.065];
	  \draw (0.0, 0.7) -- (1.05, 0.875);
	  \fill (0.0, 0.35) circle[radius=0.065];
	  \draw (0.0, 0.35) -- (1.05, 0.525);
	  \fill (1.05, 0.525) circle[radius=0.065];
	  \fill (0.0, 0.0) circle[radius=0.065];
	  \draw (0.0, 0.0) -- (1.05, 0.525); 
	  \draw (0.525, -0.35) node {$f$};
	\end{scope}
 
   \draw (1.925, 0.7) node {$=$};
   
   	\begin{scope}[shift={(2.8, 0.0)}]
	  \fill (0.0, 1.4) circle[radius=0.065];
	  \draw (0.0, 1.4) -- (1.05, 1.05);
	  \fill (1.05, 1.05) circle[radius=0.065];
	  \fill (0.0, 1.05) circle[radius=0.065];
	  \draw (0.0, 1.05) -- (1.05, 0.7);
	  \fill (1.05, 0.7) circle[radius=0.065];
	  \fill (0.0, 0.7) circle[radius=0.065];
	  \draw (0.0, 0.7) -- (1.05, 0.7);
	  \fill (0.0, 0.35) circle[radius=0.065];
	  \draw (0.0, 0.35) -- (1.05, 0.35);
	  \fill (1.05, 0.35) circle[radius=0.065];
	  \fill (0.0, 0.0) circle[radius=0.065];
	  \draw (0.0, 0.0) -- (1.05, 0.35); 
	  \draw (0.525, -0.35) node {$a$};
	  \fill (2.1, 1.05) circle[radius=0.065];
	  \draw (2.1, 1.05) -- (3.15, 0.875);
	  \fill (3.15, 0.875) circle[radius=0.065];
	  \fill (2.1, 0.7) circle[radius=0.065];
	  \draw (2.1, 0.7) -- (3.15, 0.875); 
	  \fill (2.1, 0.35) circle[radius=0.065];
	  \draw (2.1, 0.35) -- (3.15, 0.525);
	  \fill (3.15, 0.525) circle[radius=0.065];
	  \draw (2.625, -0.35) node {$b$};
	  \begin{scope}[line width=0.2mm]
		\draw[dotted] (1.05, 1.05) -- (2.1, 1.05);
		\draw[dotted] (1.05, 1.05) -- (2.1, 0.7);
		\draw[dotted] (1.05, 1.05) -- (2.1, 0.35);
		\draw[dotted] (1.05, 0.7) -- (2.1, 1.05);
		\draw[dotted] (1.05, 0.7) -- (2.1, 0.7);
		\draw[dotted] (1.05, 0.7) -- (2.1, 0.35);
		\draw[dotted] (1.05, 0.35) -- (2.1, 1.05);
		\draw[dotted] (1.05, 0.35) -- (2.1, 0.7);
		\draw[dotted] (1.05, 0.35) -- (2.1, 0.35);
	  \end{scope}
	  \draw (1.575, -0.378) node {$\varphi$};
	\end{scope}
  \end{tikzpicture}	
\end{center}
\noindent
Among the six possible bijections $\varphi : \un 3 \simeq \un 3$ 
in the picture, only 
\
\begin{tikzpicture}
  \draw (0.336, 0.336) -- (0.672, 0.336);
  \draw (0.336, 0.224) -- (0.672, 0.224);
  \draw (0.336, 0.112) -- (0.672, 0.112);
\end{tikzpicture}
,
\begin{tikzpicture}
  \draw (0.336, 0.336) -- (0.672, 0.336);
  \draw (0.336, 0.224) -- (0.672, 0.112);
  \draw (0.336, 0.112) -- (0.672, 0.224);
\end{tikzpicture}
,
\begin{tikzpicture}
  \draw (0.336, 0.336) -- (0.672, 0.224);
  \draw (0.336, 0.224) -- (0.672, 0.336);
  \draw (0.336, 0.112) -- (0.672, 0.112);
\end{tikzpicture}
,
\begin{tikzpicture}
  \draw (0.336, 0.336) -- (0.672, 0.224);
  \draw (0.336, 0.224) -- (0.672, 0.112);
  \draw (0.336, 0.112) -- (0.672, 0.336);
\end{tikzpicture} \
give a composite isomorphic to $f$,
whereas  \
\begin{tikzpicture}
  \draw (0.336, 0.336) -- (0.672, 0.112);
  \draw (0.336, 0.224) -- (0.672, 0.336);
  \draw (0.336, 0.112) -- (0.672, 0.224);
\end{tikzpicture}
,
\begin{tikzpicture}
  \draw (0.336, 0.336) -- (0.672, 0.112);
  \draw (0.336, 0.224) -- (0.672, 0.224);
  \draw (0.336, 0.112) -- (0.672, 0.336);
\end{tikzpicture} \
do not.
Corollary~\ref{sectioncoeff!} now
gives the section coefficient
$$
c^f_{a,b} = 
\frac{  3!2! \cdot 4}{2!2!2! \cdot 2!} = 3 .
$$

\medskip

\ourmargin{Deleted\\ cocycle\\ discussion}

%
%

\begin{blanko}{`Zeroth section coefficients': the counit.}
 Let us also say a word about the zeroth section coefficients, i.e.~the
 computation of the counit: the main case is when $X$ is complete (in the sense
 that $s_0$ is a monomorphism).  In this case, clearly we have
  $$
  \varepsilon(f) = \begin{cases}
    1 & \text{ if $f$ degenerate } \\
    0 & \text{ else.}
  \end{cases}
  $$
  If $X$ is Rezk complete, the first condition is equivalent to being invertible.

  The other easy case is when $X_0\simeq1$.  In this case
  $$
  \varepsilon(f) = \begin{cases}
    |\Omega(X_1,f)| & \text{ if $f$ degenerate } \\
    0 & \text{ else.}
  \end{cases}
  $$
\end{blanko}

\begin{blanko}{Example.}
  \ourmargin{Corrected}
  The strict nerve of a $1$-category $\CC$ is a decomposition space which is
  discrete in each degree.  The resulting coalgebra at the numerical level
  (assuming the due finiteness conditions) is the coalgebra of
  Content--Lemay--Leroux~\cite{Content-Lemay-Leroux}, and if the category is
  just a poset, that of Rota et al.~\cite{JoniRotaMR544721}.

  For the fat nerve $X$ of $\CC$, we find
  \label{lem:h*h}
$$
h^a * h^b \simeq \sum_{d_0 a\stackrel\varphi\to d_1 b} h^{a\varphi b} ,
$$
as follows readily by dualising Corollary~\ref{sectioncoeff!}. This formula
has a clear intuitive content: to multiply two arrows in the `fat category
algebra', sum over all connecting isos $\varphi$ to make $a$ and $b$
composable, and return the triple composite. (It should be stressed here
that the cardinality of $h^a$ is not the basis element $\delta^a$ dual to
$\delta_a$. Rather $\norm{h^a} = \norm{\Omega(X_1,a)} \delta^a$. With this, at
the numerical level the convolution algebra has $\delta^a * \delta^b =
\sum_f c^f_{a,b} \, \delta^f$, so that the section coefficients are the
same as for the coalgebra, as expected. The subtleties regarding this
duality are carefully treated in \cite{GKT:HLA}.)
\end{blanko}

\begin{blanko}{Finite support.}\label{fin-supp}
  It is also interesting to consider the subalgebra of the incidence algebra
  consisting of functions with finite support, i.e.~the full subcategory
  $\grpd^{X_1}_{\finsup} \subset \grpd^{X_1}$, and numerically $\Q^{\pi_0
  X_1}_{\finsup} \subset \Q^{\pi_0 X_1}$.  Of course we have canonical
  identifications $\grpd^{X_1}_{\finsup} \simeq \grpd_{/X_1}$, as well as
  $\Q^{\pi_0 X_1}_{\finsup} \simeq \Q_{\pi_0 X_1}$, but it is important to keep
  track of which side of duality we are on.
  
  That the decomposition space is locally finite is not the appropriate condition
  for these subalgebras to
  exist.  Instead, for the convolution product to descend to
  functors with finite support, the requirement is that $X_1$ be locally finite and
  the functor
  $$
  X_2 \to X_1 \times X_1
  $$ be finite.  (This is always the case for a locally finite Segal $1$-groupoid, by 
  Lemma~\ref{lem:Omega(y)}.)
  Similarly, one can ask for the convolution unit to have finite support,
  which is to require $X_0 \to 1$ to be a finite map.
 
  Dually, the same conditions ensure that comultiplication and counit extend
  from $\grpd_{/X_1}$ to $\Grpd_{/X_1}^{\relfin}$, which numerically is
  some sort of vector space of summable infinite linear 
  combinations~\cite[6.8]{GKT:HLA}.
  An example of this situation is given by the bialgebra of 
  $P$-trees (actually $P$-forests)~\cite{Kock:1109.5785}, whose 
  comultiplication does extend to $\Grpd_{/X_1}^{\relfin}$, but 
  whose counit does not (as there are infinitely many $P$-forests without 
  nodes).  Importantly, this non-counital coalgebra is
  the home for the so-called Green function, an infinite (homotopy)
  sum of trees, and for the Fa\`a di Bruno formula it satisfies, which does not hold
  for any finite truncation.  See \cite{GalvezCarrillo-Kock-Tonks:1207.6404} for these 
  results.
\end{blanko}

\begin{blanko}{Examples.}\label{categoryalg}
  \ourmargin{Corrected}
  If $X$ is the strict nerve of a $1$-category $\CC$, then 
  the finite-support convolution algebra is precisely the {\em category algebra}
  of $\CC$.  (For a finite category, of course
  the two notions coincide.)  
  
  Note that the convolution unit is
  $$
  \varepsilon = \sum_x \delta^{\id_x} = \begin{cases}
    1 & \text{for id arrows} \\
    0 & \text{else,}
  \end{cases}
  $$
  the sum of all indicator functions of identity arrows, so it will be finite if
  and only if the category has only finitely many objects.
  
  In the case of the fat nerve of a $1$-category, the finiteness condition
  for comultiplication is
  implied by the condition that every object has a finite automorphism
  group (a condition implied by local finiteness).
  On the other hand,
  the convolution unit has finite support precisely when there is only a finite
  number of isoclasses of objects, already a more drastic condition.
  Note the `category algebra' interpretation: arrows are multiplied by
  summing over all ways they can be made strictly composable (cf.~\ref{lem:h*h}):
$$
h^a * h^b \simeq \sum_{d_0 a\stackrel\varphi\to d_1 b} h^{a\varphi b} .
$$
%
%
%
\end{blanko}

\section{\M decomposition spaces}

\label{sec:M}

We finally come to the M\"obius condition, which ensures
the M\"obius inversion principle descends to the numerical 
level.

Recall that $\grpd$ denotes the $\infty$-category of finite $\infty$-groupoids,
as defined in \ref{bl:finiteness}.

\begin{lemma}\label{lem:comp-finite}
  If $X$ is a complete decomposition space then the following conditions are
  equivalent
  \begin{enumerate}
  \item $d_1:X_2\to X_1$ is finite.

  \item $d_1:\nondeg X_2\to X_1$ is finite.

  \item $d_1^{r-1}:\nondeg X_r\to X_1$ is finite for all $r\geq 2$.
  \end{enumerate}
\end{lemma}
\begin{proof}
  We show the first two conditions are equivalent; the third is similar.
  Using the word notation of \ref{w} we consider the map
  $$
  \nondeg X_2 + \nondeg X_1 + \nondeg X_1 + X_0\xrightarrow{~\simeq~}
  \nondeg X_2 + X_{0a} + X_{a0} + X_{00} \xrightarrow{~=~}X_2\xrightarrow{~d_1~}X_1
  $$
  Thus $d_1:X_2\to X_1$ is finite if and only if the restriction of this map to
  the first component, $d_1:\nondeg X_2\to X_1$, is finite.  By completeness the
  restrictions to the other components are finite (in fact, mono).
\end{proof}

\begin{cor}
  A complete decomposition space $X$ is locally finite if and only if $X_1$ is
  locally finite and $d_1^{r-1}:\nondeg X_r\to X_1$ is finite for all $r\geq 2$.
\end{cor}

\begin{blanko}{\M condition.}\label{M}
  A complete decomposition space $X$ is called {\em \M}
  if it is locally finite and \FILT (i.e.~of locally finite length).
  It then follows that  the restricted composition map
  $$
  \sum_r{d_1}^{r-1}:\sum_r \nondeg X_r\to X_1
  $$ 
  is finite.  In other words, the spans defining $\Phieven$ and $\Phiodd$ are of
  finite type, and hence descend to the finite slices $\grpd_{/X_1}$.
  In fact we have:
\end{blanko}

\begin{lemma}\label{lem:oldcharacterisationofM}
  A complete decomposition space $X$ is \M if and only if $X_1$ is locally
  finite and the restricted composition map
  $$
  \sum_r{d_1}^{r-1}:\sum_r \nondeg X_r\to X_1
  $$ 
  is finite. 
\end{lemma}
  
\begin{proof}
  `Only if' is clear.  Conversely, if the map $m:\sum_r{d_1}^{r-1}:\sum_r
  \nondeg X_r\to X_1$ is finite, in particular for each individual $r$ the map
  $\nondeg X_r\to X_1$ is finite, and then also $X_r \to X_1$ is finite, by
  Lemma~\ref{lem:comp-finite}.  Hence $X$ is altogether locally finite.  But it
  also follows from finiteness of $m$ that for each $a\in X_1$, the fibre
  $(\nondeg X_r)_a$ must be empty for big enough $r$, so the filtration
  condition is satisfied, so altogether $X$ is \M.
\end{proof}

\begin{BM}
  If $X$ is a Segal space, the \M condition says that for each arrow $a\in X_1$,
  the factorisations of $a$ into nondegenerate $a_i\in \nondeg X_1$ have
  bounded length.   In particular, if $X$ is the strict nerve of a $1$-category, then it is \M in the sense of
  the previous definition if and only if it is \M in the sense of 
  Leroux~\cite{Leroux:1975}.  (Note
  however that this would also have been true if we had not included the
  condition that $X_1$ be locally finite (as obviously this is automatic for any
  discrete set).  We insist on including the condition $X_1$ locally finite
  because it is needed in order to have a well-defined cardinality.)
\end{BM}

\begin{blanko}{Filtered coalgebras in vector spaces.}
  A \M decomposition space is in particular length-filtered.
  The coalgebra filtration (\ref{coalgebrafiltGrpd}) at the objective level
  $$
  \Grpd_{/X_1^{(0)}} \to \Grpd_{/X_1^{(1)}} \to \cdots \to \Grpd_{/X_1} 
  $$
  is easily seen to descend to $\grpd$-coefficients (finite 
  $\infty$-groupoids):
    $$
  \grpd_{/X_1^{(0)}} \to \grpd_{/X_1^{(1)}} \to \cdots \to \grpd_{/X_1}  ,
  $$
  and taking cardinality then yields a coalgebra filtration at
  the numerical level too.
  From the arguments in \ref{coalgebrafiltGrpd}, it follows that this
  coalgebra filtration
  $$
  C_0 \into C_1 \into \cdots \into C
  $$
  has the property that $C_0$ is generated by group-like elements. (This property
  has been found useful in the context of perturbative
  renormalisation~\cite{Kock:1411.3098}, \cite{Kock:1512.03027}, where it serves
  as a basis for recursive arguments, as an alternative to the more
  common assumption of connectedness.)  Finally, if $X$ is a graded \M
  decomposition space, then the resulting coalgebra at the algebraic level is
  furthermore a graded coalgebra.
\end{blanko}

The following is an immediate corollary to Lemma~\ref{lem:Rezk}.
It
extends the classical fact that a \M category in the sense of Leroux does not
have non-identity invertible arrows~\cite[Lemma~2.4]{LawvereMenniMR2720184}.

\begin{cor}\label{cor:MSegal=Rezk}
  Every \M decomposition space $X$ is Rezk complete.
\end{cor}

\begin{blanko}{\M inversion at the algebraic level.}
  Assume $X$ is a locally finite complete decomposition space.
  The span
  $
  \xymatrix@R-15pt{
   X_1  & \ar[l]_=  X_1\ar[r] &  1}
  $
  defines the zeta functor (cf.~\ref{zeta}), which as a presheaf is 
  $\zeta=\int^t h^t$, the homotopy sum of the representables.
  Its cardinality is the usual zeta 
  function in the incidence algebra $\Q^{\pi_0X_1}$.

  The spans $\xymatrix@R-15pt{
   X_1  & \ar[l]  \nondeg X_r\ar[r] &  1}$ define the Phi functors
  \begin{eqnarray*}
    \Phi_r: \Grpd_{/ X_1} & \longrightarrow & \Grpd ,
  \end{eqnarray*}
  with $\Phi_0=\varepsilon$. By Lemma~\ref{lem:comp-finite},
  these functors descend to 
    \begin{eqnarray*}
    \Phi_r: \grpd_{/ X_1} & \longrightarrow & \grpd ,
  \end{eqnarray*}
 and we can take cardinality to obtain functions
$
\norm{\zeta}: \pi_0(X_1) \to \Q
$
and $\norm{\Phi_r} : \pi_0(X_1) \to \Q$, elements in the incidence algebra 
$\Q^{\pi_0X_1}$.

Finally, when $X$ is furthermore assumed to be \M,
we can take cardinality of the abstract \M
inversion formula of Theorem \ref{thm:zetaPhi}:
\end{blanko}

\begin{thm}\label{thm:|M|}
  If $X$ is a \M decomposition space,
  then the cardinality of the zeta functor, $\norm{\zeta}:\Q_{\pi_0 X_1}\to\Q$,
  is convolution invertible with inverse $\norm{\mu}:= \norm{\Phieven} - \norm{\Phiodd}$:
  $$
  \norm\zeta * \norm\mu = \norm\varepsilon = \norm\mu * \norm\zeta .
  $$
\end{thm}

\begin{blanko}{Example (continued).}
  We have seen that the decomposition space $\GG$ of finite graphs of
  Example~\ref{ex:graphsdecomp} is complete, tight, and locally finite,
  (Examples \ref{ex:graphscomplete}, \ref{ex:graphstight},
  and~\ref{ex:graphslocfin}, respectively).  Hence it is a M\"obius
  decomposition space.  The general M\"obius inversion formula $\mu =
  \Phieven - \Phiodd$ yields a M\"obius inversion formula in the chromatic
  Hopf algebra, but at the numerical level this is not the most economical.
  The well-known cancellation-free formula $\mu(G) = (-1)^n$, where $n$ is
  the number of vertices of $G$, can be established by exploiting the \culf
  functor to the decomposition space of finite sets mentioned at the end of
  Example~\ref{ex:graphsdecomp}.  This argument works more generally, for
  any decomposition space arising as a restriction
  species~\cite{GKT:restriction}.
\end{blanko}

\end{document}